\documentclass[11pt]{article}
\usepackage{amsmath, amsfonts, amsthm, amssymb}  % Some math symbols
\usepackage{fullpage}

\usepackage[x11names, rgb]{xcolor}
\usepackage{graphicx}
\usepackage{tikz}
\usetikzlibrary{decorations,arrows,shapes}

\usepackage{hyperref}
\usepackage{etoolbox}
\usepackage{enumerate}
\usepackage{listings}
\usepackage{booktabs}
\usepackage{apacite}
\usepackage{comment}
\usepackage[affil-it]{authblk}
\usepackage{float}
\usepackage{algorithm2e}
\usepackage{mathtools}

\def\indented#1{\list{}{}\item[]}
\let\indented=\endlist

\newtheorem{theorem}{Theorem}[section]
\newtheorem{lemma}[theorem]{Lemma}
\newtheorem{proposition}[theorem]{Proposition}

\newtheorem{definition}{Definition}[section]

\title{Variance bounds in product measures without exponential tails}
\author{Shi Feng}
\date{October 2025}

\begin{document}

\maketitle

\begin{abstract}
    We establish analogs of Cheeger’s inequality for probability measures with heavy tails.  
    As one of the principal applications, suppose $\lambda > 3$ and  define the (Pareto) probability measure $\mu_{\lambda}$ on $[1, \infty)$ by  
    \begin{align*}
        d\mu_{\lambda}(x) = (\lambda - 1) x^{-\lambda}.
    \end{align*}
    Let $\mu_{\lambda}^n$ denote the product measure of $\mu_{\lambda}$ on $\mathbb{R}^n$.  
    Then, for any $1$-Lipschitz function (with respect to the Euclidean distance) $f : \mathbb{R}^n \to \mathbb{R}$, we obtain the variance bound  
    \begin{align*}
        \operatorname{Var}_{\mu_{\lambda}^n}(f)
        \leq C(\lambda) \, n^{\frac{2}{\lambda - 1}},
    \end{align*}
    where $C(\lambda)$ is an explicit constant depending only on $\lambda$.  
    This improves upon the existing bound $\operatorname{Var}_{\mu_{\lambda}^n}(f) = O(n)$ derived from the Efron–Stein inequality.  
    Moreover, this bound is asymptotically tight when considering the $1$-Lipschitz function $f(x) = |x|_{\infty}$ corresponding to the $L^{\infty}$ norm.  
    
    In probabilistic terms, suppose $X_1, \dots, X_n$ are i.i.d.\ random variables with distribution $\mu_{\lambda}$.  
    Then, for any $1$-Lipschitz function $f$, we have  
    \begin{align*}
        \operatorname{Var}(f(X_1, \dots, X_n))
        \leq C'(\lambda)\operatorname{Var}(\max\{X_1, \dots, X_n\})
        = \Theta\!\left(n^{\frac{2}{\lambda - 1}}\right),
    \end{align*}
    where $C'(\lambda)$ is another explicit constant depending only on $\lambda$.
\end{abstract}

\section{Introduction}

The study of functional inequalities has seen substantial growth in recent years, with applications spanning probability theory, partial differential equations, geometry, and theoretical computer science, among others.  
Good general references include~\cite{bakry2013analysis},~\cite{boucheron2013concentration},~\cite{klartag2024isoperimetric}, and~\cite{o2014analysis}.  
While functional inequalities are well understood for probability measures that are log-concave, satisfy a Poincaré inequality, or even a stronger logarithmic Sobolev inequality, considerably less is known for measures that fail to satisfy these conditions or that possess heavy tails.  

Among the various structural properties of probability measures, one that is particularly intriguing to probabilists is the existence of dimension-free variance bounds for $1$-Lipschitz functions, especially under product measures.  
One motivation for studying the variance of $1$-Lipschitz functions in high-dimensional settings is that the Kannan–Lovász–Simonovits (KLS) conjecture is equivalent to proving that the variance of any $1$-Lipschitz function under an isotropic log-concave measure is bounded by $O(1)$~\cite[Corollary~21]{klartag2024isoperimetric}.  
Therefore, it is natural to investigate the structural properties of $1$-Lipschitz functions in more tractable settings, where one may hope to gain insight into this broader phenomenon.
Another interesting book discussing the interplay between variance bounds in Gaussian measures and functional inequalities is~\cite{chatterjee2014superconcentration}.  
The interest in product measures arises because they correspond to i.i.d.\ random variables.  
If such a bound could be established, one would obtain statements of the form: given i.i.d.\ random variables $X_1, \dots, X_n$ and a $1$-Lipschitz function $f$,  
\begin{align*}
    \mathrm{Var}(f(X_1, \dots, X_n)) = O(1).
\end{align*}
However, it is well known that dimension-free variance bounds do not hold in general.  
A classical counterexample is the maximum function,
\begin{align*}
    f(X_1, \dots, X_n) = \max\{X_1, \dots, X_n\}.
\end{align*}
It is well known that if $X_i \sim \mu_{\lambda}$ are i.i.d.\ Pareto random variables as defined in the abstract, then the distribution of  
$n^{-\frac{1}{\lambda-1}} f$ converges to the Fréchet distribution; see~\cite[Theorem~3.3.7]{embrechts2013modelling}.  
As a consequence, we have
\begin{align*}
    \operatorname{Var}(f) \geq O\!\left(n^{\frac{2}{\lambda-1}}\right).
\end{align*}

In this paper, as outlined in the abstract, we establish optimal dimension-dependent variance bounds for a class of product measures that are neither log-concave nor satisfy Poincaré-type inequalities.  
Interestingly, the maximum function is precisely the one that achieves nearly the maximal variance among all $1$-Lipschitz functions.

In the existing literature, variance bounds for heavy-tailed probability measures are typically obtained through either the weak Poincaré inequality~\cite{rockner2001weak, barthe2005concentration,bobkov2007large,bobkov2009distributions} or the weighted Poincaré inequality~\cite{bobkov2009weighted}.  
A comprehensive summary of these approaches, presented within a general Lyapunov function framework, can be found in~\cite{cattiaux2010functional}. However, there are some unsatisfactory requirements for both weak Poincaré inequality and weighted Poincaré inequality. The weak Poincaré inequality requires the function $f$ to be bounded to give meaningful variance bounds. On the other hand,  weighted Poincaré inequality requires the Cauchy type law of the underlying probability measures, which is generally not applicable to product measures.

In this paper, we propose a different approach to bounding the variance of heavy-tailed probability measures.  
We begin by introducing the $\alpha$-Cheeger constant~\eqref{alpha Cheeger constant}, which serves to characterize heavy tails in the same way that the traditional Cheeger constant characterizes sub-exponential tails.  
We then show that a finite $\alpha$-Cheeger constant implies a finite $\alpha$-Poincaré constant—possibly with different values of $\alpha$—as stated in Theorem~\ref{L^2 theorem}, in direct analogy with Cheeger’s classical isoperimetric inequality.  
Finally, by tensorizing the probability measure, we obtain optimal dimension-dependent variance bounds (Theorem~\ref{product theorem}).  

The key innovation of our approach lies in exploiting the $1$-Lipschitz property to connect the fluctuation of $1$-Lipschitz functions to their associated quantile functions (Lemma~\ref{dQ lemma}).  
Working on the quantile space then allows us to perform precise integration and truncation, yielding the technical lemmas required for the proof of Theorem~\ref{L^2 theorem}.

We now begin the rigorous setup of the paper.  
Let $\nu$ be a probability measure on $\mathbb{R}^n$, and let $f: \mathbb{R}^n \to \mathbb{R}$ be a measurable function. We define the expectation of $f$ with respect to $\nu$ by  
\begin{align*}
    \mathbb{E}_{\nu}[f] = \int_{\mathbb{R}^n} f \, d\nu,
\end{align*}  
the variance of $f$ under $\nu$ by  
\begin{align*}
    \operatorname{Var}_{\nu}(f) = \mathbb{E}_{\nu}[f^2] - \big(\mathbb{E}_{\nu}[f]\big)^2,
\end{align*}  
and the median of $f$ under $\nu$ by
\begin{align*}
    m_{\nu}(f) = \inf\left\{x\in \mathbb{R}\;|\; \mathbb{P}_{\nu}(f\leq x) \geq \frac{1}{2}\right\}.
\end{align*}
For example, we can write $\operatorname{Var}_{\nu}(|x|_2)$ to represent $\operatorname{Var}_{\nu}(f)$ for $f(x) = |x|_2$ as the $L^2$ norm of $x\in \mathbb{R}^n$. Then we define set of Lipschitz function in $\mathbb{R}^n$ as
\begin{align*}
    \operatorname{Lip}(\mathbb{R}^n) = \left\{f: \mathbb{R}^n\to \mathbb{R}: \frac{|f(x)-f(y)|}{|x-y|_2}<\infty \;\forall x,y \in \mathbb{R}^n\right\}
\end{align*}
and the set of 1-Lipschitz function in $\mathbb{R}^n$ as
\begin{align*}
    \operatorname{Lip}_1(\mathbb{R}^n) = \left\{f: \mathbb{R}^n\to \mathbb{R}: \frac{|f(x)-f(y)|}{|x-y|_2}\leq 1 \;\forall x,y \in \mathbb{R}^n\right\}.
\end{align*}
Now, we define the Poincaré constant $C_P(\nu)$ as
\begin{align*}
    C_P(\nu) = \sup_{f\in \operatorname{Lip}(\mathbb{R}^n)} \frac{\operatorname{Var}_{\nu}(f)}{\mathbb{E}_{\nu}[|\nabla f|^2]}.
\end{align*}
Here $|\nabla f(x)|$ is interpreted as
\begin{align*}
    |\nabla f(x)| = \limsup_{y\to x} \frac{|f(x)-f(y)|}{|x-y|}.
\end{align*}
Then we define the Cheeger (isoperimetric) constant $I_C(\nu)$ as
\begin{align*}
    I(\nu) = \sup_{A\subseteq \mathbb{R}^n} \frac{\min \{\nu(A), 1-\nu(A)\}}{\nu(A^+)},
\end{align*}
where $\nu(A^+)$ is given by
 \begin{align*}
        \nu(A^+) = \liminf_{\epsilon\to 0} \frac{\nu(A^{\epsilon})-\nu(A)}{\epsilon}
\end{align*}
and $A^{\epsilon} = \{x\in \mathbb{R}^n\;|\; d(x,A)\leq \epsilon\}$. Cheeger’s inequality states that \cite[Proposition 13]{klartag2024isoperimetric}
\begin{align}\label{Cheeger's inequality}
    C_P(\nu) \leq 4 I(\nu)^2.
\end{align}  
However, not every probability measure admits a finite Cheeger or Poincaré constant. For instance, if $\nu$ does not have exponentially decaying tails—such as the Pareto distribution introduced in the abstract—then one has  
\begin{align*}
    C_P(\nu) = I(\nu) = \infty
\end{align*}  
(see \cite[Section 2.4]{klartag2024isoperimetric}). In this paper, we establish analogs of Cheeger’s inequality \eqref{Cheeger's inequality} that applies to probability measures without sub-exponential tails. To this end, we introduce different versions of the Poincaré and Cheeger constants.  

\begin{definition}  
For $0 < \alpha \leq 1$, the $\alpha$-Poincaré constant $\mathcal{C}_2(\nu,\alpha) \in [0,\infty]$ of a probability measure $\nu$ on $\mathbb{R}^n$ is defined by  
\begin{align}\label{alpha Poincaré constant}
    \mathcal{C}_2(\nu,\alpha) := \sup_{f \in \operatorname{Lip}_1(\mathbb{R}^n)} \frac{\mathbb{E}_{\nu}[|f - m_{\nu}(f)|^2]}{\big(\mathbb{E}_{\nu}[|\nabla f|^2]\big)^{\alpha}}.
\end{align}  

Moreover, for $0 < \alpha \leq 1$ and $1 \leq \beta < \frac{\alpha}{1-\alpha}$, the $(\alpha,\beta)$-Poincaré constant $\mathcal{C}_1(\nu,\alpha,\beta)$ is defined as  
\begin{align}\label{alpha beta Poincaré constant}
    \mathcal{C}_1(\nu,\alpha,\beta) := \sup_{f \in \operatorname{Lip}_1(\mathbb{R}^n)} \frac{\mathbb{E}_{\nu}[|f - m_{\nu}(f)|^{\beta}]}{\big(\mathbb{E}_{\nu}[|\nabla f|]\big)^{\alpha}}.
\end{align}  

Finally, for $0 < \alpha \leq 1$, the $\alpha$-Cheeger constant $\mathcal{I}(\nu,\alpha)$ is given by  
\begin{align}\label{alpha Cheeger constant}
    \mathcal{I}(\nu,\alpha) := \sup_{A \subseteq \mathbb{R}^n} \frac{\min \{\nu(A), \, 1 - \nu(A)\}}{\nu(A^+)^{\alpha}}.
\end{align}  
\end{definition}  

Here are some immediate observations concerning the definitions.  
For $\frac{2}{3}<\alpha\leq 1$, we have
$\mathcal{C}_1(\nu,\alpha,2) \geq \mathcal{C}_2(\nu,\frac{\alpha}{2})$ since $\mathbb{E}_{\nu}[|\nabla f|^2]^{\frac{\alpha}{2}} \geq \mathbb{E}_{\nu}[|\nabla f|]^{\alpha}$ by Jensen's inequality.  
If $\mathcal{C}_2(\nu,\alpha) < \infty$, then for any $1$-Lipschitz function $f$ we obtain
\begin{align*}
    \mathrm{Var}_{\nu}(f)
    \leq \mathbb{E}_{\nu}[|f - m_{\nu}(f)|^2]
    \leq \mathcal{C}_2(\nu,\alpha)\,\mathbb{E}_{\nu}[|\nabla f|^2]^{\alpha} \leq \mathcal{C}_2(\nu,\alpha).
\end{align*}
Moreover, if $0 < \alpha_1 \leq \alpha_2 \leq 1$, then
$\mathcal{C}_2(\nu,\alpha_1) \leq \mathcal{C}_2(\nu,\alpha_2)$.
Indeed, when $f$ is $1$-Lipschitz, we have $|\nabla f|^2 \leq 1$, and therefore
$\mathbb{E}_{\nu}[|\nabla f|^2]^{\alpha_1} \geq \mathbb{E}_{\nu}[|\nabla f|^2]^{\alpha_2}$.
Similarly, for any $1 \leq \beta < \frac{\alpha_1}{1-\alpha_1}$, we have
$\mathcal{C}_1(\nu,\alpha_1,\beta) \leq \mathcal{C}_1(\nu,\alpha_2,\beta)$.  
Finally, note that when $\alpha = 1$, we recover the usual Cheeger constant, since
$\mathcal{I}(\nu,1) = I(\nu)$.

Readers may recall \cite[Theorem~1.2]{bobkov2007large} in connection with the definition of the $\alpha$-Cheeger constant.  
Indeed, $\kappa$-concave probability measures form a rich class for which the condition $\mathcal{I}(\nu,\alpha)<\infty$ holds.  
We say that a probability measure $\nu$ on $\mathbb{R}^n$ is $\kappa$-concave for $-\infty<\kappa\leq 1$ if for any Borel sets $A,B\subseteq\mathbb{R}^n$ and any $0<t<1$,
\begin{align*}
    \nu(tA + (1-t)B) \leq \left(t\,\nu(A)^{\kappa} + (1-t)\,\nu(B)^{\kappa}\right)^{\frac{1}{\kappa}},
\end{align*}
where $tA + (1-t)B := \{ta+(1-t)b : a\in A,\; b\in B\}$.  
The next statement restates \cite[Theorem~1.2]{bobkov2007large} in our notation, showing that the $\alpha$-Cheeger constant is nontrivial for this class.

\begin{proposition}\label{kappa concave proposition}
    Suppose $\nu$ is a non-degenerate $\kappa$-concave probability measure on $\mathbb{R}^n$ for some $-\infty<\kappa\leq 0$.  
    Denote $m := m_{\nu}(|x|)$. Then
    \begin{align}
        \mathcal{I}\!\left(\nu, \frac{1}{1-\kappa}\right) \leq \left(\frac{c(\kappa)}{m}\right)^{\frac{1}{1-\kappa}},
    \end{align}
    where $c(\kappa)$ is a positive continuous function on $(-\infty,0]$.
\end{proposition}
In what follows we phrase our results in terms of $\mathcal{I}(\nu,\alpha)$.  
This proposition can be combined with the theorems below whenever applicable.

Throughout the remainder of this paper, we adopt the following conventions: $\nu$ will typically denote a probability measure on $\mathbb{R}^n$, while $\mu$ will denote a probability measure on $\mathbb{R}$. We use $C_i(\cdot)$ to represent important structural constant functions, whereas $c_i(\cdot)$ will be reserved for auxiliary technical constant functions.

With these conventions in place, we are ready to state the main theorems of this work. We begin with the case of the $(\alpha,\beta)$-Poincaré constant.  

\begin{theorem}\label{L^1 theorem}
    Suppose $\mathcal{I}(\nu,\alpha)<\infty$ for some $\frac{1}{2}<\alpha< 1$. Then for any 1-Lipschitz function $f$ and $1\leq \beta < \frac{\alpha}{1-\alpha}$, there exists constant $C_1(\alpha,\beta,\mathcal{I}(\nu,\alpha))<\infty$ such that
    \begin{align}
        \mathbb{E}_{\nu}[|f-m_{\nu}(f)|^{\beta}] \leq C_1(\alpha,\beta,\mathcal{I}(\nu,\alpha))\cdot \mathbb{E}_{\nu}[|\nabla f|]^{\alpha-\beta(1-\alpha)}.
    \end{align}
    Here we can take
    \begin{align}
        C_1(\alpha,\beta,\mathcal{I}(\nu,\alpha)) = \mathcal{I}(\nu,\alpha)^{\alpha-\beta(1-\alpha)} 
        \left[1+2\mathcal{I}(\nu,\alpha)^{\frac{\beta}{\alpha}} \cdot \left(\frac{\alpha}{1-\alpha}\right)^{\beta}\cdot \left(\frac{\alpha}{\alpha-\beta(1-\alpha)}\right)\right].
    \end{align}
    Equivalently, we have the following bound for the $(\alpha,\beta)$-Poincaré constant
    \begin{align}
        \mathcal{C}_1(\nu,\alpha-\beta(1-\alpha),\beta) \leq C_1(\alpha,\beta,\mathcal{I}(\nu,\alpha)).
    \end{align}
\end{theorem}

Intuitively, the theorem states that by restricting the class of functions to $1$-Lipschitz functions, one can control all moments of the function entirely in terms of its gradient.  
Note, however, that this bound is generally not asymptotically tight; see, for instance, Theorem~\ref{L^1 Cheeger theorem} for comparison.  
By “not asymptotically tight,” we mean that there may exist some $\alpha' > \alpha - \beta(1 - \alpha)$ such that  
$\mathcal{I}(\nu, \alpha) < \infty$ still implies $\mathcal{C}_1(\nu, \alpha', \beta) < \infty$.

However, Theorem \ref{L^1 theorem} provides a general statement in which the parameter $\beta$ can be chosen with considerable flexibility.  
Moreover, its proof is slightly simpler and serves as an introduction to the techniques developed in this paper.  
For these reasons, it is the first theorem that we present and prove.

We emphasize the similarities and differences between Theorem~\ref{L^1 theorem} and~\cite[Theorem~5.1 and Lemma~5.2]{bobkov2009distributions}.  
Both works use $\mathcal{I}(\nu,\alpha)$ (or closely related isoperimetric information) to control fluctuations of functions under $\nu$.  
A key difference is that~\cite{bobkov2009distributions} does not impose a $1$-Lipschitz condition on $f$.  
In that setting, the resulting moment-type estimates naturally live in the regime $\beta<2$, while the borderline case $\beta=2$---corresponding to variance---is typically the most relevant from a probabilistic perspective.

In contrast, Theorem~\ref{L^1 theorem} imposes the additional assumption that $f$ is $1$-Lipschitz.  
This structural restriction allows one to exploit the geometry of the underlying space more effectively and leads to higher moments bounds that are inaccessible without Lipschitz regularity.  
In particular, the $1$-Lipschitz condition enables control at and beyond the variance scale $\beta=2$, which cannot be achieved in general in the absence of such a condition.

We now present what is arguably the most important and technically demanding theorem of this paper.  
It provides an asymptotically tight bound on $\mathcal{C}_2(\nu,\cdot)$ in terms of $\mathcal{I}(\nu,\alpha)$. The tightness of this result will be discussed in Section~6.

\begin{theorem}\label{L^2 theorem}
    Suppose $\mathcal{I}(\nu,\alpha)<\infty$ for some $\frac{2}{3}<\alpha < 1$. Then for any 1-Lipschitz function $f$, there exists constant $C_2(\alpha,\mathcal{I}(\nu,\alpha))<\infty$ such that
    \begin{align}
        \mathbb{E}_{\nu}[|f-m_{\nu}(f)|^2] \leq C_2(\alpha,\mathcal{I}(\nu,\alpha))\cdot \mathbb{E}_{\nu}[|\nabla f|^2]^{\frac{3\alpha-2}{\alpha}}.
    \end{align}
    Here we can take
    \begin{align}
        C_2(\alpha,\mathcal{I}(\nu,\alpha)) = 2^{\frac{16\alpha-10}{\alpha}}\mathcal{I}(\nu,\alpha)^{\frac{6\alpha-4}{\alpha}}\left[2\mathcal{I}(\nu,\alpha)^{\tfrac{2}{\alpha}}\left(\frac{2\alpha^2}{(3\alpha-2)(1-\alpha)}\right) \right]^{\frac{2(1-\alpha)}{\alpha}}.
    \end{align}
    Equivalently, we have the following bound for the $\alpha$-Poincaré constant 
    \begin{align}
        \mathcal{C}_2\left(\nu, \frac{3\alpha-2}{\alpha}\right) \leq C_2(\alpha,\mathcal{I}(\nu,\alpha)).
    \end{align}
\end{theorem}

By an argument very similar to, and in fact simpler than, that of Theorem~\ref{L^2 theorem},  
we can establish an $L^1$ version of the result.  
This theorem refines Theorem~\ref{L^1 theorem} in the special case $\beta = 1$.

\begin{theorem} \label{L^1 Cheeger theorem}
    Suppose $\mathcal{I}(\nu,\alpha)<\infty$ for some $\frac{1}{2}<\alpha< 1$. Then for any 1-Lipschitz function $f$, there exists constant $C_3(\alpha,\mathcal{I}(\nu,\alpha))<\infty$ such that 
    \begin{align}
        \mathbb{E}_{\nu}[|f-m_{\nu}(f)|] \leq C_3(\alpha,\mathcal{I}(\nu,\alpha))\cdot \mathbb{E}_{\nu}[|\nabla f|]^{\frac{2\alpha-1}{\alpha}}.
    \end{align}
    Here we can take
    \begin{align}
        C_3(\alpha,\mathcal{I}(\nu,\alpha)) = 2(2\mathcal{I}(\nu,\alpha))^{\frac{2\alpha-1}{\alpha}}\left[2\mathcal{I}(\nu,\alpha)^{\tfrac{1}{\alpha}}\left(\frac{\alpha}{2\alpha-1}\right) \right]^{\frac{1-\alpha}{\alpha}}.
    \end{align}
    Equivalently, we have the following bound for the $(\alpha,1)$-Poincaré constant
    \begin{align}
        \mathcal{C}_1\left(\nu,\frac{2\alpha-1}{\alpha},1\right) \leq C_3(\alpha,\mathcal{I}(\nu,\alpha)).
    \end{align}

\end{theorem}

To see the refinement to Theorem~\ref{L^1 theorem}, note that since $\alpha \leq 1$, we have $\frac{2\alpha - 1}{\alpha} \geq 2\alpha - 1 = \alpha-\beta(1-\alpha)$ for $\beta=1$. Consequently,
\begin{align*}
    \mathcal{C}_1(\nu, 2\alpha - 1, 1)
    \leq \mathcal{C}_1\left(\nu, \frac{2\alpha - 1}{\alpha}, 1\right).
\end{align*}
Moreover, in Section~6 we will also claim that Theorem \ref{L^1 Cheeger theorem} is asymptotically tight.

We now turn to applications of Theorem~\ref{L^2 theorem}.  
In the existing literature, for a product measure $\mu^n$ such that either $C_P(\mu) < \infty$ or $I(\mu) < \infty$, it is well known that for any $1$-Lipschitz function $f: \mathbb{R}^n \to \mathbb{R}$, we have  
\begin{align*}
    \mathrm{Var}_{\mu^n}(f) \leq C
\end{align*}
for some constant $C$ that does not depend on $n$.  
This can be shown either by induction on the Poincaré constant~\cite[Exercise~1]{klartag2024isoperimetric} or, by a more involved argument, through induction on the Cheeger constant~\cite{bobkov1997isoperimetric}.

In \cite{bobkov2009weighted}, the authors study probability measures in high dimensions that do not possess exponential tails, focusing in particular on Cauchy distributions and more generally on $\kappa$-concave probability measures. They establish a weighted Poincaré-type inequality stating that, for any Lipschitz function $f$,
\begin{align*}
    \mathrm{Var}_{\nu}(f) \leq C \cdot \mathbb{E}_{\nu}\!\left[\,|\nabla f(x)|^2 \cdot (1 + |x|^2)\,\right],
\end{align*}
where $\nu$ denotes the Cauchy distribution (under certain restrictions) and $C$ is a universal constant. For the precise formulation, see Theorem 3.1 in \cite{bobkov2009weighted}.

For product measures $\mu^n$ with $\mu$ having bounded support, \cite{talagrand1995concentration} investigated the fluctuations of $1$-Lipschitz functions with respect to the metric $d_1$ (defined in \eqref{d_p definition}) and established the celebrated Talagrand concentration inequality.  
In addition, McDiarmid’s inequality~\cite{boucheron2013concentration} provides another powerful tool to control the fluctuations of $1$-Lipschitz functions with respect to $d_1$.  
Both results yield sharp dimension-dependent fluctuation bounds and sub-Gaussian concentration.

For product measures with heavy tails, a weak Poincaré inequality is available~\cite[Proposition~5.31]{cattiaux2010functional} of the form  
\begin{align*}
    \operatorname{Var}_{\mu^n}(f)
    \leq
    C\left(\frac{n}{s}\right)^{\frac{2}{\alpha}}
    \mathbb{E}_{\mu^n}[|\nabla f|^2]
    + s \cdot \mathrm{Osc}_{\mu^n}(f)^2,
\end{align*}
for every $s \in (0, \tfrac{1}{4})$, where $C$ is a universal constant, $\alpha$ depends on the tail behavior of the measure $\mu$, and  
$\mathrm{Osc}_{\mu^n}(f) = \operatorname{ess}\sup(f) - \operatorname{ess}\inf(f)$ denotes the oscillation of $f$.  
However, this inequality yields meaningful bounds only when $f$ is bounded.

In our paper, as a principal application of Theorem \ref{L^2 theorem}, we establish non-trivial bounds for Lipschitz functions under product measures without exponential tails. Suppose $\mu$ is a probability measure on $\mathbb{R}$ such that $\operatorname{Var}_{\mu}(x) < \infty$, but $\mu$ does not satisfy the Poincaré inequality, i.e., $C_P(\mu) = \infty$. In this setting, the best known bound for the variance in the product measure is  
\begin{align*}
    \sup_{f \in \operatorname{Lip}_1(\mathbb{R}^n)} \operatorname{Var}_{\mu^n}(f) = O(n),
\end{align*}  
which follows from the Efron–Stein inequality \cite{efron1981jackknife}. In the following theorem, we demonstrate an improvement of the bound using the $\alpha$-Poincaré inequality.  

\begin{theorem}\label{product theorem}
    If $\mathcal{C}_2(\mu, \alpha)<\infty$ for some $0<\alpha<1$, then for any 1-Lipschitz $f$,
    \begin{align}
        \operatorname{Var}_{\mu^n}(f) \leq \mathcal{C}_2(\mu, \alpha)\cdot n^{1-\alpha}\cdot \mathbb{E}_{\mu^n}[|\nabla f|^2]^{\alpha} = O(n^{1-\alpha}).
    \end{align}
    Together with Theorem \ref{L^2 theorem}, if $\mathcal{I}(\mu, \alpha)<\infty$ for some $\frac{2}{3}<\alpha<1$, then for any 1-Lipschitz $f$,
    \begin{align}
        \operatorname{Var}_{\mu^n}(f) \leq C_2(\alpha,\mathcal{I}(\mu,\alpha))\cdot n^{1-\frac{3\alpha-2}{\alpha}}\cdot \mathbb{E}_{\mu^n}[|\nabla f|^2]^{\frac{3\alpha-2}{\alpha}} = O(n^{\frac{2(1-\alpha)}{\alpha}}).
    \end{align}
\end{theorem}

Note that, in general, verifying $\mathcal{I}(\mu,\alpha)<\infty$ is relatively straightforward when $\mu$ is a probability measure on $\mathbb{R}$.  
In this case, one can take advantage of~\cite[Lemma~2.2]{bobkov2007large}.  
A systematic approach is to first define the function $I : (0,1) \to \mathbb{R}_+$ by  
$I(x) = d\mu(Q_f(x))$,  
where $d\mu(\cdot)$ is the probability density function of $\mu$ on $\mathbb{R}$, $Q_f$ is the quantile function from Definition~\ref{Q_f definition}, and $f(x) = x$.  
Then, by~\cite[Lemma~2.2]{bobkov2007large}, one can determine whether $\mu$ is $\kappa$-concave.  
Finally, applying Proposition~\ref{kappa concave proposition} yields a bound on $\mathcal{I}(\mu,\alpha)$. Another approach to bounding $\mathcal{I}(\mu,\alpha)$ for a measure $\mu$ on $\mathbb{R}$ is illustrated in the proof of Lemma~\ref{Pareto lemma}, where the bound is obtained directly by analyzing the density and tail behavior of $\mu$.

Next, we turn our attention to the Pareto probability measure $d\mu_{\lambda} = (\lambda-1)x^{-\lambda}$ supported on $[1,\infty)$ introduced in the abstract, and present our main result in this setting. 

\begin{theorem}\label{Pareto theorem}
    For any $\lambda>3$, we have for any 1-Lipschitz function $f$
    \begin{align}
        \operatorname{Var}_{\mu_{\lambda}^n}(f) \leq C(\lambda)n^{\frac{2}{\lambda-1}},
    \end{align}
    where
    \begin{align}
        C(\lambda) = 2^{\frac{6\lambda-16}{\lambda-1}} (\lambda-1)^{-\frac{2\lambda-6}{\lambda}}\left(\frac{4}{\lambda-3}\right)^{\frac{2}{\lambda-1}}.
    \end{align}
    Therefore, we have 
    \begin{align}
        \sup_{f \in \operatorname{Lip}_1(\mathbb{R}^n)} \operatorname{Var}_{\mu_{\lambda}^n}(f) = \Theta(n^{\frac{2}{\lambda-1}}).
    \end{align}
\end{theorem}

Finally, we present a variance bound for 1-Lipschitz functions with respect to alternative metrics in product measures. For $x = (x_1, \dots, x_n)$ and $y = (y_1, \dots, y_n) \in \mathbb{R}^n$, and for $1\leq p \leq \infty$, define the distance $d_p(x,y)$ by
\begin{align}\label{d_p definition}
    d_p(x,y) = \left(\sum_{i=1}^{n} |x_i - y_i|^p \right)^{\frac{1}{p}}.
\end{align}
Note that $d_2$ coincides with the usual Euclidean distance. A function $f$ is said to be 1-Lipschitz with respect to $d_p$ (or $d_p$ 1-Lipschitz) if, for all $x, y \in \mathbb{R}^n$,
\begin{align*}
    |f(x) - f(y)| \leq d_p(x,y).
\end{align*}
Note that for $1 < p \leq \infty$, a function $f$ is $1$-Lipschitz with respect to $d_p$ if and only if (for almost everywhere with respect to the Lebesgue measure)
\begin{align*}
    \sum_{i=1}^{n} |d_i f|^{\frac{p}{p-1}} \leq 1,
\end{align*}
where $d_i f$ is the $i$th derivative of $f$:
\begin{align*}
    d_if(x_1,...,x_i,...,x_n) = \limsup_{\epsilon\to 0} \frac{f(x_1,...,x_i+\epsilon,...,x_n)-f(x_1,...,x_i,...,x_n)}{\epsilon}.
\end{align*}
When $p = \infty$, we define $\frac{p}{p-1} = 1$.  
We then obtain the following variance bound for $1$-Lipschitz functions with respect to $d_p$.  
Note that when $p = 1$, the bound follows easily from the Efron--Stein inequality, so we focus only on the case $1 < p \leq \infty$.

\begin{theorem}\label{d_p theorem}
    If $\mathcal{C}_1(\mu, \alpha,2)<\infty$ for some $0<\alpha<1$, then for any 1-Lipschitz $f$ with respect to $d_p$ for $1<p\leq \infty$, we have
    \begin{align}
        \operatorname{Var}_{\mu^n}(f) \leq \mathcal{C}_1(\mu,\alpha,2)\cdot n^{1-\frac{\alpha(p-1)}{p}} \cdot \left(\sum_{i=1}^n \mathbb{E}_{\mu^n}[|d_i f|^{\frac{p}{p-1}}]\right)^{\frac{\alpha(p-1)}{p}} = O(n^{\frac{p-\alpha(p-1)}{p}}).
    \end{align}
    Suppose $\mathcal{C}_2(\mu,\alpha)<\infty$ for some $0<\alpha<1$, then then for any 1-Lipschitz $f$ with respect to $d_p$ for $1<p\leq 2$, we have
    \begin{align}
        \operatorname{Var}_{\mu^n}(f) \leq \mathcal{C}_2(\mu,\alpha)\cdot n^{1-\frac{2\alpha(p-1)}{p}} \cdot \left(\sum_{i=1}^n \mathbb{E}_{\mu^n}[|d_i f|^{\frac{p}{p-1}}]\right)^{\frac{2\alpha(p-1)}{p}} = O(n^{\frac{p-2\alpha(p-1)}{p}}).
    \end{align}
    They can be combined with Theorem \ref{L^1 theorem} and \ref{L^2 theorem}.
\end{theorem}

The paper is organized to gradually develop the main ideas and techniques before applying them to increasingly structured settings.  
In Section~2, we establish the technical foundations of the paper.  
This section introduces the key analytic tools—most notably quantile-based estimates and truncation arguments—that underlie all subsequent proofs.  
In particular, Lemmas~\ref{truncation lemma} and~\ref{main L^2 lemma} encapsulate the core mechanism for converting geometric information into moments bounds.

Section~3 applies these tools to derive the main functional inequalities.  
We first prove Theorems~\ref{L^1 theorem} and~\ref{L^2 theorem}, which establish Cheeger-type implications for heavy-tailed measures.  
The section concludes with the proof of Theorem~\ref{L^1 Cheeger theorem}, showing how the $L^2$ argument can be adapted to an $L^1$ framework.

In Section~4, we turn to tensorization and product measures.  
Using the inequalities developed earlier, we obtain dimension-dependent variance bounds and identify the correct scaling behavior, leading to the proofs of Theorems~\ref{product theorem} and~\ref{Pareto theorem}.  
Section~5 extends these ideas to more general metrics, where we prove Theorem~\ref{d_p theorem} using an argument parallel to that of the Euclidean metric case.

Finally, Section~6 discusses consequences, examples, and extensions of our results.  
In particular, we analyze the sharpness of the obtained bounds and explain why Theorems~\ref{L^2 theorem} and~\ref{L^1 Cheeger theorem} are optimal up to constants.

\section{Technical lemmas}
In this section, we study some properties of 1-Lipschitz functions when the underlying measure $\nu$ satisfies $\mathcal{I}(\nu,\alpha)<\infty$ for some $\tfrac{1}{2}<\alpha<1$. We begin by recalling the definition of the quantile function associated with a random variable $f$ under $\nu$.

\begin{definition} \label{Q_f definition}
Let $f$ be a random variable defined on $(\mathbb{R}^n,\nu)$. The \emph{quantile function} of $f$ is the increasing function $Q_f:(0,1)\to\mathbb{R}$ defined by
\begin{align}
    Q_f(x) = \inf\{t \in \mathbb{R} \;:\; \mathbb{P}_\nu(f \leq t) \geq x\}.
\end{align}
Throughout this section, $\nu$ will be a fixed probability measure on $\mathbb{R}^n$, and we therefore omit $\nu$ from the notation of $Q_f$.
\end{definition}

Note that if $m_{\nu}(f)=0$, then $Q_f(\tfrac{1}{2})=0$. We now introduce the notion of one-sided derivatives for the quantile function. The right derivative of $Q_f$ at $x$ is defined by
\begin{align*}
    \frac{d^+}{dx}Q_f(x) = \limsup_{\epsilon \to 0} \frac{Q_f(x+\epsilon)-Q_f(x)}{\epsilon}.
\end{align*}
Similarly, the left derivative of $Q_f$ at $x$ is given by
\begin{align*}
    \frac{d^-}{dx}Q_f(x) = \limsup_{\epsilon \to 0} \frac{Q_f(x)-Q_f(x-\epsilon)}{\epsilon}.
\end{align*}
We define the derivative of $Q_f$ at $x$ to be
\begin{align*}
    \frac{d}{dx}Q_f(x) = \max\!\left\{\frac{d^+}{dx}Q_f(x), \frac{d^-}{dx}Q_f(x)\right\}.
\end{align*}

We now claim that $\tfrac{d}{dx}Q_f(x)$ can be controlled by the $\alpha$-Cheeger constant $\mathcal{I}(\nu,\alpha)$. Intuitively, $\tfrac{d}{dx}Q_f(x)$ describes the rate at which the quantile function changes, and thus quantifies how much $f$ may fluctuate under $\nu$. Therefore, the following lemma shows that the fluctuation of a 1-Lipschitz function is bounded by the $\alpha$-Cheeger constant of $\nu$. And this is the starting point of the whole argument.

\begin{lemma}\label{dQ lemma}
    Suppose $\mathcal{I}(\nu,\alpha)<\infty$ for some $0<\alpha\leq 1$ and let $f$ be a 1-Lipschitz function. Then for all $x \in (0,1)$,
    \begin{align}
        \frac{d}{dx}Q_f(x) \leq \left(\frac{\mathcal{I}(\nu,\alpha)}{\min\{x,1-x\}}\right)^{\tfrac{1}{\alpha}}.
    \end{align}
\end{lemma}

\begin{proof}
We prove the bound for $\tfrac{d^-}{dx}Q_f(x)$; the case of $\tfrac{d^+}{dx}Q_f(x)$ follows by the same argument. Define the level set
\begin{align}\label{upper level set}
    A(f,t) := \{y \in \mathbb{R}^n \;:\; f(y) \geq t\}.
\end{align}
Note that by definition we have $\nu(A(f,Q_f(x))) = 1-x$ for all $0<x<1$. Recall the definition $A^{\epsilon} = \{x\in \mathbb{R}^n\;|\; d(x,A)\leq \epsilon\}$ in the introduction section. Since $f$ is 1-Lipschitz, it follows that $A(f,t)^{\epsilon} \subseteq A(f,t-\epsilon)$ for all $t \in \mathbb{R}$ and $\epsilon \geq 0$. Consequently,
\begin{align*}
    \frac{d^-}{dx}Q_f(x) &= \limsup_{\epsilon \to 0} \frac{Q_f(x) - Q_f(x-\epsilon)}{\epsilon} \\
        &= \limsup_{\epsilon \to 0} \frac{\epsilon}{\nu(A(f,Q_f(x)-\epsilon)) - \nu(A(f,Q_f(x)))} \\
        &\leq \limsup_{\epsilon \to 0} \frac{\epsilon}{\nu(A(f,Q_f(x))^{\epsilon}) - \nu(A(f,Q_f(x)))} \\
        &= \frac{1}{\nu(A(f,Q_f(x))^+)} \\
        &\leq \left(\frac{\mathcal{I}(\nu,\alpha)}{\min\{x,1-x\}}\right)^{\tfrac{1}{\alpha}}.
\end{align*}
The key step above is the second equality, which follows from exchanging the roles of horizontal and vertical increments.  
Viewing the graph of the function $Q_f$, the derivative is usually interpreted as taking the horizontal increment to zero and measuring the corresponding vertical change.  
Equivalently, one may take the vertical increment to zero and examine the induced horizontal change.  
These two perspectives are interchangeable and lead to the same expression.

To make this precise, we only prove the $\leq$ inequality; the reverse inequality follows from the same argument.  
Consider a sequence $x_1, x_2, \ldots \to x^-$ converging to $x$ from the left.  
Then $\{Q_f(x_i)\}$ is a bounded increasing sequence and therefore converges.

If $\lim_{n\to\infty} Q_f(x_n) < Q_f(x)$, then for any  
$\epsilon < Q_f(x) - \lim_{n\to\infty} Q_f(x_n)$ we have
\begin{align*}
    \nu(A(f,Q_f(x)-\epsilon)) - \nu(A(f,Q_f(x))) = 0,
\end{align*}
and consequently
\begin{align*}
    \limsup_{\epsilon \to 0}
    \frac{\epsilon}{\nu(A(f,Q_f(x)-\epsilon)) - \nu(A(f,Q_f(x)))} = \infty.
\end{align*}

Now suppose instead that $\lim_{n\to\infty} Q_f(x_n) = Q_f(x)$.  
Then, since $\{Q_f(x_n)\}$ converges to $Q_f(x)$, we obtain
\begin{align*}
    \limsup_{\epsilon \to 0}
    \frac{\epsilon}{\nu(A(f,Q_f(x)-\epsilon)) - \nu(A(f,Q_f(x)))}
    \geq
    \limsup_{n\to\infty}
    \frac{Q_f(x) - Q_f(x_n)}{x - x_n}.
\end{align*}

Taking the supremum over all sequences $x_1, x_2, \ldots \to x^-$ yields
\begin{align*} \limsup_{\epsilon \to 0} \frac{Q_f(x) - Q_f(x-\epsilon)}{\epsilon} &= \sup_{x_1,x_2,...\to x^-} \limsup_{n\to \infty} \frac{Q_f(x)-Q_f(x_n)}{x-x_n}\\ &\leq \limsup_{\epsilon \to 0} \frac{\epsilon}{\nu(A(f,Q_f(x)-\epsilon)) - \nu(A(f,Q_f(x)))}. 
\end{align*}
\end{proof}

Given a bound on the derivative of the function $Q_f$, it is natural to expect that one can obtain a bound on the function itself by integrating the derivative.  
However, the derivative we consider here is defined in a slightly nonstandard way: we use $\limsup$ instead of the usual limit.  
Moreover, it is not immediately clear—at least in a fully rigorous sense—whether the fundamental theorem of calculus continues to hold in the form  
\begin{align}
    Q_f(t_1) - Q_f(t_2) = \int_{t_1}^{t_2} \frac{d}{dx} Q_f(x) \, dx.
\end{align}

One possible reference supporting an affirmative answer to the applicability of the fundamental theorem of calculus in our setting is~\cite[Section~6.2, Exercise~23]{royden1988real}.  
However, since this result appears only as an exercise and its derivation from the text is not entirely clear to the author, we provide a complete proof here for the sake of completeness.

\begin{lemma}\label{FTC lemma}
    Suppose $\mathcal{I}(\nu,\alpha) < \infty$ for some $0 < \alpha \leq 1$, and let $f$ be a $1$-Lipschitz function. Then, for any $0 < t_1 < t_2 < 1$, the quantile function $Q_f(x)$ is uniformly Lipschitz on $[t_1, t_2]$. Hence, if $m_{\nu}(f) = 0$ and $\tfrac{1}{2}<t<1$ (the case $0<t < \tfrac{1}{2}$ follows analogously), we obtain  
\begin{align}
    Q_f(t) &= \int_{\frac{1}{2}}^{t} \frac{d}{dx} Q_f(x) \, dx \\
    &\leq \mathcal{I}(\nu,\alpha)^{\frac{1}{\alpha}} \frac{\alpha}{1 - \alpha} (1 - t)^{1 - \frac{1}{\alpha}}.
\end{align}
\end{lemma}
\begin{proof}
    Without loss of generality, suppose $1 - t_2 < t_1$.  
    We claim that $Q_f(x)$ is uniformly Lipschitz on $[t_1, t_2]$ with constant $\left(\frac{\mathcal{I}(\nu,\alpha)}{1 - t_2}\right)^{\frac{1}{\alpha}}$.  
    We proceed by contradiction.  
    Assume that there exist $t_1 \leq a < b \leq t_2$ such that  
    \begin{align*}
        Q_f(b) - Q_f(a)
        > \left(\frac{\mathcal{I}(\nu,\alpha)}{1 - t_2}\right)^{\frac{1}{\alpha}} \cdot (b - a).
    \end{align*}
    Define $x_1 = \frac{a + b}{2}$.  
    Then either  
    \begin{align*}
        Q_f(b) - Q_f(x_1)
        > \left(\frac{\mathcal{I}(\nu,\alpha)}{1 - t_2}\right)^{\frac{1}{\alpha}} \cdot (b - x_1),
    \end{align*}
    or  
    \begin{align*}
        Q_f(x_1) - Q_f(a)
        > \left(\frac{\mathcal{I}(\nu,\alpha)}{1 - t_2}\right)^{\frac{1}{\alpha}} \cdot (x_1 - a).
    \end{align*}
    Without loss of generality, suppose the interval $[x_1, b]$ violates the Lipschitz condition.  
    We then define $x_2 = \frac{x_1 + b}{2}$.  
    Once again, either $[x_1, x_2]$ or $[x_2, b]$ contradicts the Lipschitz bound.  
    Assume $[x_1, x_2]$ is the violating interval and define $x_3 = \frac{x_1 + x_2}{2}$. If both the interval violate the Lipschitz condition, pick on the next interval arbitrarily. By continuing this process, we construct a non-repeating sequence $\{x_i\}_{i \geq 1}$ contained in $[a, b]$. Since $[a, b]$ is compact and $\{x_i\}_{i \geq 1}$ is Cauchy, there exists $x^*$ such that $x_n \to x^*$.  
    Suppose there exists $N \in \mathbb{N}$ such that either $x_n < x^*$ for all $n \geq N$, or $x_n > x^*$ for all $n \geq N$.  
    Then $x^* \in \{x_i\}_{i \geq 1} \cup \{a, b\}$.  
    Consequently, the convergent sequence $\{x_i\}_{i \geq 1} \to x^*$ contradicts Lemma~\ref{dQ lemma}.  
    
    On the other hand, suppose the sequence approaches $x^*$ from both sides.  
    Equivalently, there exists a sequence of open intervals $(x_{p(i)}, x_{q(i)})_{i \geq 1}$ such that  
    $x^* \in (x_{p(i)}, x_{q(i)})$ for all $i$,  
    $\lim_{i \to \infty} (x_{q(i)} - x_{p(i)}) = 0$,  
    and $|p(i) - q(i)| = 1$.  
    Then we have for all $i\geq 1$ 
    \begin{align*}
        \max\left\{
            \frac{Q_f(x^*) - Q_f(x_{p(i)})}{x^* - x_{p(i)}},
            \frac{Q_f(x_{q(i)}) - Q_f(x^*)}{x_{q(i)} - x^*}
        \right\}
        &\geq \frac{Q_f(x_{q(i)}) - Q_f(x_{p(i)})}{x_{q(i)} - x_{p(i)}} \\
        &\geq \left(\frac{\mathcal{I}(\nu,\alpha)}{1 - t_2}\right)^{\frac{1}{\alpha}}.
    \end{align*}
    Again, this contradicts Lemma~\ref{dQ lemma}, by considering the convergent subsequences $x_{q(i)} \to x^*$ or $x_{p(i)} \to x^*$.

    Since $Q_f(x)$ is uniformly Lipschitz on $[t_1, t_2]$, it follows that $Q_f(x)$ is absolutely continuous on $[t_1, t_2]$.  
    By~\cite[Theorem~3.11]{stein2009real}, $Q_f$ is differentiable almost everywhere on $[t_1, t_2]$, and for any $t > \tfrac{1}{2}$ we can write  
    \begin{align*}
        Q_f(t) - Q_f\!\left(\tfrac{1}{2}\right)
        &= \int_{\tfrac{1}{2}}^{t} \frac{d}{dx} Q_f(x) \, dx \\
        &\leq \int_{\tfrac{1}{2}}^{t} \left(\frac{\mathcal{I}(\nu,\alpha)}{1 - x}\right)^{\tfrac{1}{\alpha}} \, dx \\
        &= \mathcal{I}(\nu,\alpha)^{\frac{1}{\alpha}} \frac{\alpha}{1 - \alpha}
           \Big((1 - t)^{1 - \tfrac{1}{\alpha}} - \left(\tfrac{1}{2}\right)^{1 - \tfrac{1}{\alpha}}\Big) \\
        &\leq \mathcal{I}(\nu,\alpha)^{\frac{1}{\alpha}} \frac{\alpha}{1 - \alpha} (1 - t)^{1 - \frac{1}{\alpha}}.
\end{align*}

\end{proof}

We now state our first main technical lemma. We will use it to prove Theorem \ref{L^1 theorem}. At first glance, the statement may appear somewhat unintuitive. However, it plays a central role in ensuring that the arguments in Section 3 work. We will return to this lemma in Section 3, where its importance will become clear.

\begin{lemma}\label{truncation lemma}
    For any 1-Lipschitz function $f$ with $m_{\nu}(f) = 0$, define the truncated function $f_{[-\frac{1}{2},\frac{1}{2}]}$ as
    \begin{align*}
        f_{[-\frac{1}{2},\frac{1}{2}]} (x) := \frac{1}{2}\mathbb{I}\left(f(x) \geq \frac{1}{2}\right)+ f(x)\cdot \mathbb{I}\left(|f(x)|<\frac{1}{2}\right) + (-\frac{1}{2})\mathbb{I}\left(f(x)\leq -\frac{1}{2}\right),
    \end{align*}
    where $\mathbb{I}(\cdot)$ is the indicator function. Suppose $\mathcal{}\mathcal{I}(\nu,\alpha) <\infty$ for some $\frac{1}{2}<\alpha<1$, then for any $1\leq \beta <\frac{\alpha}{1-\alpha}$ and $1\leq \gamma\leq \beta$,
    \begin{align}
        \mathbb{E}_{\nu}[|f|^{\beta}] \leq c_1(\alpha,\beta,\gamma,\mathcal{I}(\nu,\alpha))\mathbb{E}_{\nu}[|(f^{\gamma})_{[-\frac{1}{2},\frac{1}{2}]}|]^{1-\frac{\beta(1-\alpha)}{\alpha}},
    \end{align}
    where
    \begin{align}
        c_1(\alpha,\beta,\gamma,\mathcal{I}(\nu,\alpha)) = 1+\mathcal{I}(\nu,\alpha)^{\frac{\beta}{\alpha}} \cdot \left(\frac{\alpha}{1-\alpha}\right)^{\beta}\cdot \left(\frac{\alpha}{\alpha-\beta(1-\alpha)}\right)\cdot2^{\gamma-\frac{(\gamma-1)\beta(1-\alpha)}{\alpha}}.
    \end{align}
\end{lemma}

\begin{proof}
    Suppose $Q_f(p_1) = \tfrac{1}{2}$ and $Q_f(p_2) = -\tfrac{1}{2}$ for some $p_2 \leq \tfrac{1}{2} \leq p_1$.
    Then we can bound the following integral by Lemma \ref{FTC lemma}:            
    \begin{align*}
        \int_{p_1}^{1} Q_f(x)^{\beta}\,dx 
            &\leq \int_{p_1}^{1}\left(\mathcal{I}(\nu,\alpha)^{\tfrac{1}{\alpha}}\frac{\alpha}{1-\alpha}(1-x)^{1-\tfrac{1}{\alpha}}\right)^{\beta}dx \\
            &= \mathcal{I}(\nu,\alpha)^{\frac{\beta}{\alpha}} \cdot \left(\frac{\alpha}{1-\alpha}\right)^{\beta}\cdot \int_{p_1}^1 (1-x)^{-\frac{\beta(1-\alpha)}{\alpha}}dx\\
            &= \mathcal{I}(\nu,\alpha)^{\frac{\beta}{\alpha}} \cdot \left(\frac{\alpha}{1-\alpha}\right)^{\beta}\cdot \left(\frac{\alpha}{\alpha-\beta(1-\alpha)}\right)\cdot (1-p_1)^{\frac{\alpha - \beta(1-\alpha)}{\alpha}}.
    \end{align*}

    For the last equality, we use the fact that $0<\frac{\beta(1-\alpha)}{\alpha}<1$, so we can evaluate the integral explicitly. Note the same argument we used for the upper tail integral also applies to the lower tail. Thus,
    \begin{align*}
        \int_{0}^{p_2}|Q_f(x)|^{\beta} dx \leq \mathcal{I}(\nu,\alpha)^{\frac{\beta}{\alpha}} \cdot \left(\frac{\alpha}{1-\alpha}\right)^{\beta}\cdot \left(\frac{\alpha}{\alpha-\beta(1-\alpha)}\right)\cdot p_2^{\frac{\alpha - \beta(1-\alpha)}{\alpha}}.
    \end{align*}
    Since the quantile function $Q_f$ under $\mathrm{Unif}(0,1)$ has the same distribution as $f$ under $\nu$, we also know that $|((Q_f)^\gamma)_{[-1/2,1/2]}|$ under $\mathrm{Unif}(0,1)$ has the same distribution as $|(f^\gamma)_{[-1/2,1/2]}|$ under $\nu$. Hence,
    \begin{align*}
        \mathbb{E}_{\nu}[|(f^{\gamma})_{[-\frac{1}{2},\frac{1}{2}]}|]  = \int_{0}^{1} |(Q_f(x)^{\gamma})_{[-\frac{1}{2},\frac{1}{2}]}| dx.
    \end{align*}

    Therefore, we have
    \begin{align*}
        \mathbb{E}_{\nu}[|f|^{\beta}]-\mathbb{E}[|(f^{\gamma})_{[-\frac{1}{2},\frac{1}{2}]}|]  &= \int_{0}^{1} |Q_f(x)|^{\beta} - |(Q_f(x)^{\gamma})_{[-\frac{1}{2},\frac{1}{2}]}| dx\\
        &= \int_{0}^{p_2}|Q_f(x)|^{\beta} - |(Q_f(x)^{\gamma})_{[-\frac{1}{2},\frac{1}{2}]}| dx + \int_{p_2}^{p_1}|Q_f(x)|^{\beta} - |(Q_f(x)^{\gamma})_{[-\frac{1}{2},\frac{1}{2}]}| dx\\ 
        &+ \int_{p_1}^{1}|Q_f(x)|^{\beta} - |(Q_f(x)^{\gamma})_{[-\frac{1}{2},\frac{1}{2}]}| dx\\
        &\leq \int_{0}^{p_2}|Q_f(x)|^{\beta}dx + \int_{p_2}^{p_1}|Q_f(x)|^{\beta} - |(Q_f(x)^{\gamma})_{[-\frac{1}{2},\frac{1}{2}]}| dx + \int_{p_1}^{1}|Q_f(x)|^{\beta} dx\\
    \end{align*}
    Since $1\leq \gamma\leq\beta$ and $|Q_f(x)| \leq \frac{1}{2}$ for $x\in[p_1,p_2]$, we have
    \begin{align*}
        |Q_f(x)|^{\beta} \leq \left|(Q_f(x))^{\gamma}\right| = \left|(Q_f(x))^{\gamma}_{[-\frac{1}{2},\frac{1}{2}]}\right|
    \end{align*}
    for all $x\in [p_1,p_2]$, which implies the middle term is non-positive. Hence, we have
    \begin{align*}
        \mathbb{E}_{\nu}[|f|^{\beta}]-\mathbb{E}[|(f^{\gamma})_{[-\frac{1}{2},\frac{1}{2}]}|] 
        &\leq \int_{0}^{p_2}|Q_f(x)|^{\beta} dx + \int_{p_1}^{1} |Q_f(x)|^{\beta} dx\\
        &\leq \mathcal{I}(\nu,\alpha)^{\frac{\beta}{\alpha}} \cdot \left(\frac{\alpha}{1-\alpha}\right)^{\beta}\cdot \left(\frac{\alpha}{\alpha-\beta(1-\alpha)}\right)\cdot \left(p_2^{1-\frac{\beta(1-\alpha)}{\alpha}} + (1-p_1)^{1-\frac{\beta(1-\alpha)}{\alpha}}\right)
    \end{align*}
    Since $\beta<\frac{\alpha}{1-\alpha}$, we have $1-\frac{\beta(1-\alpha)}{\alpha}\in (0,1)$. Now, by concavity of $x \mapsto x^{1-\frac{\beta(1-\alpha)}{\alpha}}$,
    \begin{align*}
        p_2^{1-\frac{\beta(1-\alpha)}{\alpha}} + (1-p_1)^{1-\frac{\beta(1-\alpha)}{\alpha}} \leq 2\left(\frac{p_2+1-p_1}{2}\right)^{1-\frac{\beta(1-\alpha)}{\alpha}} = 2^{\frac{\beta(1-\alpha)}{\alpha}} (p_2+1-p_1)^{1-\frac{\beta(1-\alpha)}{\alpha}}.
    \end{align*}
    Bringing it back to the above inequality, we get
    \begin{align*}
        \mathbb{E}_{\nu}[|f|^{\beta}]-\mathbb{E}[|(f^{\gamma})_{[-\frac{1}{2},\frac{1}{2}]}|] 
        \leq \mathcal{I}(\nu,\alpha)^{\frac{\beta}{\alpha}} \cdot \left(\frac{\alpha}{1-\alpha}\right)^{\beta}\cdot \left(\frac{\alpha}{\alpha-\beta(1-\alpha)}\right)\cdot2^{\frac{\beta(1-\alpha)}{\alpha}}\cdot(p_2+1-p_1)^{1-\frac{\beta(1-\alpha)}{\alpha}}
    \end{align*}
    Note that on $(0,p_2) \cup(p_1,1)$, we have $(\frac{1}{2})^{\gamma} \leq |(f^{\gamma})_{[-\frac{1}{2},\frac{1}{2}]}(x)| \leq \frac{1}{2}$. Thus, we have
    \begin{align*}
        2^{\gamma}\mathbb{E}_{\nu}[|(f^{\gamma})_{[-\frac{1}{2},\frac{1}{2}]}(x)|] \geq (p_2+1-p_1)
    \end{align*}
    and
    \begin{align*}
        \mathbb{E}_{\nu}[|(f^{\gamma})_{[-\frac{1}{2},\frac{1}{2}]}|] \leq \mathbb{E}_{\nu}[|(f^{\gamma})_{[-\frac{1}{2},\frac{1}{2}]}|]^{1-\frac{\beta(1-\alpha)}{\alpha}}.
    \end{align*}
    Putting everything together yields the final bound:
    \begin{align*}
        \mathbb{E}_{\nu}[|f|^{\beta}] &\leq \mathbb{E}_{\nu}[|(f^{\gamma})_{[-\frac{1}{2},\frac{1}{2}]}|] + \mathcal{I}(\nu,\alpha)^{\frac{\beta}{\alpha}} \cdot \left(\frac{\alpha}{1-\alpha}\right)^{\beta}\cdot \left(\frac{\alpha}{\alpha-\beta(1-\alpha)}\right)\cdot2^{\frac{\beta(1-\alpha)}{\alpha}}\cdot(2^{\gamma}\mathbb{E}_{\nu}[|(f^{\gamma})_{[-\frac{1}{2},\frac{1}{2}]}|])^{1-\frac{\beta(1-\alpha)}{\alpha}}\\
        &\leq \left[1+\mathcal{I}(\nu,\alpha)^{\frac{\beta}{\alpha}} \cdot \left(\frac{\alpha}{1-\alpha}\right)^{\beta}\cdot \left(\frac{\alpha}{\alpha-\beta(1-\alpha)}\right)\cdot2^{\gamma-\frac{(\gamma-1)\beta(1-\alpha)}{\alpha}}\right]\cdot(\mathbb{E}_{\nu}[|(f^{\gamma})_{[-\frac{1}{2},\frac{1}{2}]}|])^{1-\frac{\beta(1-\alpha)}{\alpha}}.
    \end{align*}
    
\end{proof}

The proof of Theorem~\ref{L^1 theorem} relies only on Lemma~\ref{truncation lemma}. Therefore, we suggest that readers skip ahead to Section~3 to better understand the motivation for this type of estimate. Doing so may also make the subsequent lemmas for Theorems~\ref{L^2 theorem} and~\ref{L^1 Cheeger theorem} more intuitive. Nevertheless, for coherence, we prefer to present all technical lemmas together in Section~2.

Having gained some experience with the computation of quantile functions, we now turn to a more technically demanding result required for the proof of Theorem~\ref{L^2 theorem} and \ref{L^1 Cheeger theorem}.  
Before delving into the computations, we first clarify the type of estimates we aim to obtain.  
We describe them here in an intuitive, though informal, manner; full rigor will be restored in the proof itself.  

As emphasized in the proof of Theorem~\ref{L^1 theorem} in Section~3, we wish to avoid using inequalities of the form
\begin{align*}
    \int_{0}^{\infty}\mathbb{P}_{\nu}(f \geq t)^{\frac{1}{\alpha}}\,dt
    \geq
    \int_{0}^{\frac{1}{2}}\mathbb{P}_{\nu}(f \geq t)^{\frac{1}{\alpha}}\,dt.
\end{align*}
Instead, our goal is to establish estimates of the form
\begin{align*}
    \int_{0}^{\infty}\mathbb{P}_{\nu}(f \geq t)^{\frac{1}{\alpha}}\,dt
    \geq ? \cdot
    \left(\int_{0}^{\infty}\mathbb{P}_{\nu}(f \geq t)\,dt\right)^{?}
    \approx ? \cdot \mathbb{E}_{\nu}[|f|]^{?}.
\end{align*}
In other words, we seek to determine the optimal exponents and constants represented by the question marks above.  
This is precisely the objective of Lemma~\ref{main L^2 lemma}, where we establish the corresponding bounds.

To this end, we begin by introducing a family of functions and establishing a key property that they satisfy.
\begin{lemma}\label{boundary function lemma}
    Suppose $\mathcal{I}(\nu,\alpha)<\infty$ for some $\frac{1}{2}<\alpha<1$ and $f$ is 1-Lipschitz with $m_{\nu}(f) = 0$. Then for any $1\leq \gamma <\frac{\alpha}{1-\alpha}$, define $G_{\gamma}: [\frac{1}{2},1) \to \mathbb{R}_{+}$ as
    \begin{align}
        G_{\gamma}(x) = \mathcal{I}(\nu,\alpha)^{\tfrac{\gamma}{\alpha}}\left(\frac{\alpha}{1-\alpha}\right)^{\gamma}(1-x)^{-\frac{\gamma(1-\alpha)}{\alpha}}.
    \end{align}
    Then we will have for any $x\in [\frac{1}{2},1)$,
    \begin{align}
        \frac{d}{dx}G_{\gamma}(x) \geq \gamma \cdot Q_f(x)^{\gamma-1}\cdot \frac{d}{dx}Q_f(x) = \frac{d}{dx}Q_f^{\gamma}(x)
    \end{align}
\end{lemma}
\begin{proof}
    We have established upper bounds for $ Q_f(x) $ in Lemma \ref{FTC lemma} and for its derivative $ \frac{d}{dx}Q_f(x) $ in Lemma~\ref{dQ lemma}. Combining these results will yield the desired proof. 
    For the left-hand side, by the definition of $ G_{\gamma}$, we obtain
    \begin{align*}
        \frac{d}{dx}G_{\gamma}(x) 
        &= \mathcal{I}(\nu,\alpha)^{\tfrac{\gamma}{\alpha}}\left(\frac{\alpha}{1-\alpha}\right)^{\gamma}(1-x)^{-\frac{\gamma(1-\alpha)}{\alpha}-1}\cdot \frac{\gamma(1-\alpha)}{\alpha}\\
        &= \gamma \cdot \mathcal{I}(\nu,\alpha)^{\tfrac{\gamma}{\alpha}}\left(\frac{\alpha}{1-\alpha}\right)^{\gamma-1}(1-x)^{-\frac{\gamma-\gamma\alpha + \alpha}{\alpha}}.
    \end{align*}
    For the right hand side, we have
    \begin{align*}
        \gamma \cdot Q_f(x)^{\gamma-1}\cdot \frac{d}{dx}Q_f(x) 
        &\leq \gamma\cdot \mathcal{I}(\nu,\alpha)^{\tfrac{\gamma-1}{\alpha}}\left(\frac{\alpha}{1-\alpha}\right)^{\gamma-1}(1-x)^{-\tfrac{(\gamma-1)(1-\alpha)}{\alpha}} \cdot \left(\frac{\mathcal{I}(\nu,\alpha)}{1-x}\right)^{\frac{1}{\alpha}}\\
        &= \gamma \cdot \mathcal{I}(\nu,\alpha)^{\tfrac{\gamma}{\alpha}}\left(\frac{\alpha}{1-\alpha}\right)^{\gamma-1}(1-x)^{-\frac{\gamma-\gamma\alpha + \alpha}{\alpha}}\\
        &= \frac{d}{dx}G_{\gamma}(x)
    \end{align*}
\end{proof}

Since the derivative of $Q_f^{\gamma}$ is bounded by $G_{\gamma}$, we have the following tail bound for $Q_f^{\gamma}$.

\begin{lemma} \label{Q_f^gamma tail bound lemma}
    Suppose $\mathcal{I}(\nu,\alpha)<\infty$ for some $\frac{1}{2}<\alpha<1$ and $f$ is 1-Lipschitz with $m_{\nu}(f) = 0$. Then for any $1\leq \gamma <\frac{\alpha}{1-\alpha}$ and $\frac{1}{2}\leq p <1$, we have
    \begin{align}
        \int_{p}^{1} Q_f(x)^\gamma - Q_f(p)^{\gamma}dx 
        &\leq \int_{p}^{1} G_{\gamma}(x) - G_{\gamma}(p)dx \\
        &= \mathcal{I}(\nu,\alpha)^{\tfrac{\gamma}{\alpha}}\left(\frac{\alpha}{1-\alpha}\right)^{\gamma} \left(\frac{\gamma(1-\alpha)}{\alpha+\gamma\alpha-\gamma}\right) \cdot(1-p)^{\frac{\alpha+\gamma\alpha-\gamma}{\alpha}}\\
        &:= c_2(\alpha,\gamma,\mathcal{I}(\nu,\alpha))\cdot(1-p)^{\frac{\alpha+\gamma\alpha-\gamma}{\alpha}}
    \end{align}
\end{lemma}
\begin{proof}
    The first inequality follows directly from Lemma \ref{boundary function lemma}. The second inequality comes from direct computation. Since $-1<-\frac{\gamma(1-\alpha)}{\alpha}<0$, we have
    \begin{align*}
        \int_{p}^{1} G_{\gamma}(x) - G_{\gamma}(p)dx 
        &= \int_{p}^{1} \mathcal{I}(\nu,\alpha)^{\tfrac{\gamma}{\alpha}}\left(\frac{\alpha}{1-\alpha}\right)^{\gamma}(1-x)^{-\frac{\gamma(1-\alpha)}{\alpha}} dx -\mathcal{I}(\nu,\alpha)^{\tfrac{\gamma}{\alpha}}\left(\frac{\alpha}{1-\alpha}\right)^{\gamma}(1-p)^{-\frac{\gamma(1-\alpha)}{\alpha}+1}\\
        &= \mathcal{I}(\nu,\alpha)^{\tfrac{\gamma}{\alpha}}\left(\frac{\alpha}{1-\alpha}\right)^{\gamma} (1-p)^{-\frac{\gamma(1-\alpha)}{\alpha}+1} \cdot \left(\frac{1}{-\frac{\gamma(1-\alpha)}{\alpha}+1}-1 \right)\\
        &= \mathcal{I}(\nu,\alpha)^{\tfrac{\gamma}{\alpha}}\left(\frac{\alpha}{1-\alpha}\right)^{\gamma} (1-p)^{\frac{\alpha+\gamma\alpha-\gamma}{\alpha}} \cdot \left(\frac{\gamma(1-\alpha)}{\alpha+\gamma\alpha-\gamma}\right).
    \end{align*}
\end{proof}

We then give the following technical definition. It intuitively means the point that achieves half of the mass. Its importance will be shown in the proof of our main technical Lemma \ref{main L^2 lemma}.

\begin{definition} \label{half-mass point definition}
    Given an increasing continuous function $G: [\frac{1}{2},1) \to \mathbb{R}_+$ such that $G(\frac{1}{2}) = 0$ and
    \begin{align}
        \int_{\frac{1}{2}}^1 G(x) dx  = M<\infty.
    \end{align}
    define $p(G) \in [\frac{1}{2},1)$ be the smallest number such that
    \begin{align}
        \int_{\frac{1}{2}}^{1} \min\{G(x),G(p(G))\}dx=\int_{p(G)}^1 G(x)-G(p(G)) dx = \frac{M}{2}.
    \end{align}
\end{definition}

Note that $p(G)$ always exists and is unique since $G$ is continuous and increasing. With Lemma \ref{Q_f^gamma tail bound lemma}, we can give an upper bound for $p(Q_f^{\gamma})$.

\begin{lemma}\label{half mass bound lemma}
    Suppose $\mathcal{I}(\nu,\alpha)<\infty$ for some $\frac{1}{2}<\alpha<1$ and $f$ is 1-Lipschitz with $m_{\nu}(f) = 0$. Then we have for any $1\leq \gamma < \frac{\alpha}{1-\alpha}$,
    \begin{align}
        \int_{\frac{1}{2}}^1 Q_f(x)^{\gamma}dx=M<\infty.
    \end{align}
    Moreover, we have
    \begin{align}
        1-p(Q_f^{\gamma}) \geq c_3(\alpha,\gamma,\mathcal{I}(\nu,\alpha))M^{\frac{\alpha}{\alpha+\gamma\alpha-\gamma}}.,
    \end{align}
    where
    \begin{align}
        c_3(\alpha,\gamma,\mathcal{I}(\nu,\alpha)) 
        &= [2c_2(\alpha,\gamma,\mathcal{I}(\nu,\alpha))]^{-\frac{\alpha}{\alpha+\gamma\alpha-\gamma}}\\
        &= \left[2\mathcal{I}(\nu,\alpha)^{\tfrac{\gamma}{\alpha}}\left(\frac{\alpha}{1-\alpha}\right)^{\gamma} \left(\frac{\gamma(1-\alpha)}{\alpha+\gamma\alpha-\gamma}\right) \right]^{-\frac{\alpha}{\alpha+\gamma\alpha-\gamma}}.
    \end{align}
\end{lemma}
\begin{proof}
    $\int_{\frac{1}{2}}^1 Q_f(x)^{\gamma}dx<\infty$ follows directly from Lemma \ref{FTC lemma} and the monotone convergence theorem. For the second inequality, by Lemma \ref{Q_f^gamma tail bound lemma}, we have 
    \begin{align*}
        \frac{M}{2} = \int_{p(Q_f^{\gamma})}^{1} Q_f(x)^\gamma - Q_f(p(Q_f^{\gamma}))^{\gamma}dx\leq c_2(\alpha,\gamma,\mathcal{I}(\nu,\alpha)) \cdot(1-p(Q_f^{\gamma}))^{\frac{\alpha+\gamma\alpha-\gamma}{\alpha}} .
    \end{align*}
    Since $\frac{\alpha+\gamma\alpha-\gamma}{\alpha}>0$, exchanging the terms we have
    \begin{align*}
        1-p(Q_f^{\gamma}) \geq [2c_2(\alpha,\gamma,\mathcal{I}(\nu,\alpha))]^{-\frac{\alpha}{\alpha+\gamma\alpha-\gamma}} \cdot M^{\frac{\alpha}{\alpha+\gamma\alpha-\gamma}}
    \end{align*}
\end{proof}

Finally, we can state our main lemma to prove Theorem \ref{L^2 theorem} and \ref{L^1 Cheeger theorem}.

\begin{lemma}\label{main L^2 lemma}
    Suppose $\mathcal{I}(\nu,\alpha)<\infty$ for some $\frac{1}{2}<\alpha<1$ and $f$ is 1-Lipschitz with $m_{\nu}(f) = 0$. Then for any $1\leq \gamma <\frac{\alpha}{1-\alpha}$, suppose 
    \begin{align}
        \int_{0}^{\infty}\mathbb{P}_{\nu}(sign(f) \cdot |f^{\gamma}|\geq t)\,dt = M.
    \end{align}
    Then for any $\beta\geq 1$, we have
    \begin{align}
        \int_{0}^{\infty}\mathbb{P}_{\nu}(sign(f) \cdot |f^{\gamma}|\geq t)^{\beta}\,dt \geq c_4(\alpha,\beta,\gamma,\mathcal{I}(\nu,\alpha))M^{\frac{\alpha\beta+\gamma\alpha-\gamma}{\alpha+\gamma\alpha-\gamma}},
    \end{align}
    where
    \begin{align}
        c_4(\alpha,\beta,\gamma,\mathcal{I}(\nu,\alpha)) &= \frac{c_3(\alpha,\gamma,\mathcal{I}(\nu,\alpha))^{\beta-1}}{2}\\
        &= \frac{1}{2}\left[2\mathcal{I}(\nu,\alpha)^{\tfrac{\gamma}{\alpha}}\left(\frac{\alpha}{1-\alpha}\right)^{\gamma} \left(\frac{\gamma(1-\alpha)}{\alpha+\gamma\alpha-\gamma}\right) \right]^{-\frac{\alpha(\beta-1)}{\alpha+\gamma\alpha-\gamma}}.
    \end{align}
    Here $sign(f)=1$ if $f\geq 0$ and $sign(f)=-1$ if $f <0$.
\end{lemma}
\begin{proof}
    Note that by Definition \ref{half-mass point definition},
    \begin{align*}
        \int_{0}^{Q_f^{\gamma}(p(Q_f^\gamma))}\mathbb{P}_{\nu}(sign(f) \cdot |f^{\gamma}|\geq t)\,dt = \frac{M}{2}.
    \end{align*}
    Therefore, together with Lemma \ref{half mass bound lemma} we have
    \begin{align*}
        \int_{0}^{\infty}\mathbb{P}_{\nu}(sign(f)\cdot |f^{\gamma}|\geq t)^{\beta}\,dt 
        &\geq \int_{0}^{Q_f^{\gamma}(p(Q_f^\gamma))}\mathbb{P}_{\nu}(sign(f) \cdot|f^{\gamma}|\geq t)^{\beta}\,dt\\
        &\geq \int_{0}^{Q_f^{\gamma}(p(Q_f^\gamma))}\mathbb{P}_{\nu}(sign(f)\cdot |f^{\gamma}|\geq t)\,dt \cdot (1-p(Q_f^{\gamma}))^{\beta-1}\\
        &\geq  \frac{M}{2} \cdot \left(c_3(\alpha,\gamma,\mathcal{I}(\nu,\alpha))M^{\frac{\alpha}{\alpha+\gamma\alpha-\gamma}}\right)^{\beta-1}\\
        &\geq \frac{c_3(\alpha,\gamma,\mathcal{I}(\nu,\alpha))^{\beta-1}}{2}M^{\frac{\alpha\beta+\gamma\alpha-\gamma}{\alpha+\gamma\alpha-\gamma}}
    \end{align*}
    For the second inequality, we use the property that $\mathbb{P}_{\nu}(sign(f)\cdot |f^{\gamma}|\geq t)$ is a decreasing function with respect to $t$ and 
    \begin{align*}
        \mathbb{P}_{\nu}(sign(f)\cdot |f^{\gamma}|\geq Q_f^{\gamma}(p(Q_f^{\gamma}))) = 1-p(Q_f^{\gamma})
    \end{align*}
\end{proof}

\section{Proof of Theorems \ref{L^1 theorem}, \ref{L^2 theorem}, and \ref{L^1 Cheeger theorem}}
In this section, we prove Theorem \ref{L^1 theorem}, \ref{L^2 theorem}, and \ref{L^1 Cheeger theorem} following the strategy presented in \cite[Section 2.1]{klartag2024isoperimetric}. However, a key obstacle arises during the proof, and this is precisely where Lemma \ref{truncation lemma} and \ref{main L^2 lemma} play a crucial role.

\subsection{Proof of Theorem \ref{L^1 theorem}}
\begin{proof}
    Without loss of generality, we assume that $m_{\nu}(f) = 0$. Recall the level set $A(f,t)$ defined in \eqref{upper level set}. We begin with the co-area formula \cite[Lemma 3.2]{bobkov1997isoperimetric}:
    \begin{align*}
        \mathbb{E}_{\nu}[|\nabla f|] &\geq \int_{\mathbb{R}} \nu(A(f,t)^+) \,dt\\
        &\geq \frac{1}{\mathcal{I}(\nu,\alpha)}\int_{\mathbb{R}} \min\{\nu(A(f,t)),1-\nu(A(f,t))\}^{\frac{1}{\alpha}} \,dt\\
        &= \frac{1}{\mathcal{I}(\nu,\alpha)}\left(\int_{-\infty}^{0} \mathbb{P}_{\nu}(f\leq t)^{\frac{1}{\alpha}}\,dt + \int_{0}^{\infty}\mathbb{P}_{\nu}(f\geq t)^{\frac{1}{\alpha}}\,dt\right)\\
        &\geq \frac{1}{\mathcal{I}(\nu,\alpha)}\left(\int_{-\frac{1}{2}}^{0} \mathbb{P}_{\nu}(f\leq t)^{\frac{1}{\alpha}}\,dt + \int_{0}^{\frac{1}{2}}\mathbb{P}_{\nu}(f\geq t)^{\frac{1}{\alpha}}\,dt\right)
    \end{align*}
    The second inequality follows from the definition of the $\alpha$-Cheeger constant (see Definition \eqref{alpha Cheeger constant}).  
    We emphasize the last inequality: restricting the domain of integration from $\mathbb{R}$ to $[-\frac{1}{2},\frac{1}{2}]$ leads to a considerable loss, and this is reason we develop Lemma \ref{main L^2 lemma} to avoid this loss in special cases. Nevertheless, this restriction enables the application of Jensen’s inequality over the uniform measure on $[-\frac{1}{2},\frac{1}{2}]$, which allows us to extract the exponent $\frac{1}{\alpha}$ from the integral:
    \begin{align*}
        \int_{-\frac{1}{2}}^{0} \mathbb{P}_{\nu}(f\leq t)^{\frac{1}{\alpha}}\,dt + \int_{0}^{\frac{1}{2}}\mathbb{P}_{\nu}(f\geq t)^{\frac{1}{\alpha}}\,dt 
        &\geq \left(\int_{-\frac{1}{2}}^{0} \mathbb{P}_{\nu}(f\leq t)\,dt + \int_{0}^{\frac{1}{2}}\mathbb{P}_{\nu}(f\geq t)\,dt\right)^{\frac{1}{\alpha}}\\
        &= \mathbb{E}_{\nu}[|f_{[-\frac{1}{2},\frac{1}{2}]}|]^{\frac{1}{\alpha}}.
    \end{align*}

    Now Lemma \ref{truncation lemma} reveals its significance. This is precisely the reason why we define the truncation $f_{[-\frac{1}{2},\frac{1}{2}]}$ in this particular way, rather than using other possible truncations. By Lemma \ref{truncation lemma}, we have
    \begin{align*}
        \mathbb{E}_{\nu}[|f_{[-\frac{1}{2},\frac{1}{2}]}|] 
        \geq c_1(\alpha,\beta,1,\mathcal{I}(\nu,\alpha))^{-\frac{\alpha}{\alpha-\beta(1-\alpha)}} 
        \, \mathbb{E}_{\nu}[|f|^\beta]^{\frac{\alpha}{\alpha-\beta(1-\alpha)}}.
    \end{align*}
    Using this estimate in the previous inequality gives
    \begin{align*}
        \mathbb{E}_{\nu}[|\nabla f|] 
        &\geq \frac{1}{\mathcal{I}(\nu,\alpha)} 
        \, \mathbb{E}_{\nu}[|f_{[-\frac{1}{2},\frac{1}{2}]}|]^{\frac{1}{\alpha}}\\
        &\geq \frac{1}{\mathcal{I}(\nu,\alpha)} 
        \, c_1(\alpha,\beta,1,\mathcal{I}(\nu,\alpha))^{-\frac{1}{\alpha-\beta(1-\alpha)}}
        \, \mathbb{E}_{\nu}[|f|^\beta]^{\frac{1}{\alpha-\beta(1-\alpha)}}.
    \end{align*}
    Equivalently, we can express this inequality as
    \begin{align*}
        \mathbb{E}_{\nu}[|f|^\beta] 
        \leq \mathcal{I}(\nu,\alpha)^{\alpha-\beta(1-\alpha)} 
        \, c_1(\alpha,\beta,1,\mathcal{I}(\nu,\alpha)) 
        \, \mathbb{E}_{\nu}[|\nabla f|]^{\alpha-\beta(1-\alpha)}.
    \end{align*}
    Therefore, we have
    \begin{align*}
        C_1(\alpha,\beta,\mathcal{I}(\nu,\alpha)) 
        &= \mathcal{I}(\nu,\alpha)^{\alpha-\beta(1-\alpha)} 
        \, c_1(\alpha,\beta,1,\mathcal{I}(\nu,\alpha))\\
        &= \mathcal{I}(\nu,\alpha)^{\alpha-\beta(1-\alpha)} 
        \left[1+2\mathcal{I}(\nu,\alpha)^{\frac{\beta}{\alpha}} \cdot \left(\frac{\alpha}{1-\alpha}\right)^{\beta}\cdot \left(\frac{\alpha}{\alpha-\beta(1-\alpha)}\right)\right].
    \end{align*}

\end{proof}

\subsection{Proof of Theorem \ref{L^2 theorem}}
\begin{proof}
    Without loss of generality, we assume that $m_{\nu}(f) = 0$.  
We begin by applying the Cauchy--Schwarz inequality:
\begin{align}\label{Cauchy–Schwarz bound}
    \mathbb{E}_{\nu}[|\nabla (f^2)|]
    = \mathbb{E}_{\nu}[2|f||\nabla f|]
    \leq 2\sqrt{\mathbb{E}_{\nu}[f^2]\mathbb{E}_{\nu}[|\nabla f|^2]}.
\end{align}
On the other hand, the co-area formula\cite[Lemma 3.2]{bobkov1997isoperimetric} implies that
\begin{align*}
    \mathbb{E}_{\nu}[|\nabla (f^2)|]
    &\geq \int_{0}^{\infty} \nu(A(f^2,t)^+)\,dt\\
    &= \int_{\mathbb{R}} \nu(A(sign(f)\cdot  f^2,t)^+)\,dt\\
    &\geq \frac{1}{\mathcal{I}(\nu,\alpha)}\int_{\mathbb{R}}
    \min\{\nu(A(sign(f)\cdot f^2,t)),1-\nu(A(sign(f)\cdot f^2,t))\}^{1/\alpha}\,dt.
\end{align*}
For the second equality, we use the fact that $f$ is continuous and therefore for $t>0$,
\begin{align*}
    \nu(A(f^2,t)^+) = \nu(A(sign(f)\cdot  f^2,t)^+) + \nu(A(sign(f)\cdot  f^2,-t)^+).
\end{align*}
Note that since $m_{\nu}(f) = 0$, when $t\leq0$, we have $1-\nu(A(sign(f)\cdot f^2,t)) \leq \frac{1}{2} \leq  \nu(A(sign(f)\cdot f^2,t))$. Therefore, when $t\leq 0$.
\begin{align*}
    \min\{\nu(A(sign(f)\cdot f^2,t)),1-\nu(A(sign(f)\cdot f^2,t))\} = 1-\nu(A(sign(f)\cdot f^2,t)) = \mathbb{P}_{\nu}(sign(f)\cdot f^2\leq t).
\end{align*}
Similarly, when $t\geq 0$, we have
\begin{align*}
    \min\{\nu(A(sign(f)\cdot f^2,t)),1-\nu(A(sign(f)\cdot f^2,t))\} = \nu(A(sign(f)\cdot f^2,t)) = \mathbb{P}_{\nu}(sign(f)\cdot f^2\geq t).
\end{align*}
Therefore, we have
\begin{align*}
    \mathbb{E}_{\nu}[|\nabla (f^2)|]
    &\geq \frac{1}{\mathcal{I}(\nu,\alpha)}
    \left(\int_{-\infty}^{0}\mathbb{P}_{\nu}(sign(f)\cdot f^2\le t)^{1/\alpha}\,dt+\int_{0}^{\infty}\mathbb{P}_{\nu}(sign(f)\cdot f^2\ge t)^{1/\alpha}\,dt\right)\\
    &\geq \frac{1}{\mathcal{I}(\nu,\alpha)}\left(\int_{0}^{\infty}\mathbb{P}_{\nu}((f_-)^2\ge t)^{1/\alpha}\,dt+ \int_{0}^{\infty}\mathbb{P}_{\nu}((f_+)^2\ge t)^{1/\alpha}\,dt\right),
\end{align*}
where we define $f_+ := \max\{0,f\}$ and $f_- := \min\{0,f\}$.  
Instead of truncating the integral domain to $[-\tfrac{1}{2},\tfrac{1}{2}]$, we invoke Lemma~\ref{main L^2 lemma} with $\beta = 1/\alpha$ and $\gamma = 2$. This gives
\begin{align*}
    \int_{0}^{\infty}\mathbb{P}_{\nu}((f_+)^2\ge t)^{1/\alpha}\,dt
    \geq c_4\!\left(\alpha,\tfrac{1}{\alpha},2,\mathcal{I}(\nu,\alpha)\right)
    \mathbb{E}_{\nu}[(f_+)^2]^{\frac{2\alpha-1}{3\alpha-2}}.
\end{align*}
Applying the same argument to $-f$ gives an identical inequality with $f_-$ in place of $f_+$:
\begin{align*}
    \int_{0}^{\infty}\mathbb{P}_{\nu}((f_-)^2\ge t)^{1/\alpha}\,dt
    \geq c_4\!\left(\alpha,\tfrac{1}{\alpha},2,\mathcal{I}(\nu,\alpha)\right)
    \mathbb{E}_{\nu}[(f_-)^2]^{\frac{2\alpha-1}{3\alpha-2}}.
\end{align*}

Bringing them back to the inequality, we get
\begin{align*}
    \mathbb{E}_{\nu}[|\nabla (f^2)|]
    &\geq \frac{1}{\mathcal{I}(\nu,\alpha)}\,
    c_4\!\left(\alpha,\tfrac{1}{\alpha},2,\mathcal{I}(\nu,\alpha)\right)
    \left(\mathbb{E}_{\nu}[(f_-)^2]^{\frac{2\alpha-1}{3\alpha-2}}+\mathbb{E}_{\nu}[(f_+)^2]^{\frac{2\alpha-1}{3\alpha-2}}\right)\\
    &\geq \frac{1}{\mathcal{I}(\nu,\alpha)}\,
    c_4\!\left(\alpha,\tfrac{1}{\alpha},2,\mathcal{I}(\nu,\alpha)\right)
    \left(\frac{1}{2}\mathbb{E}_{\nu}[f^2]\right)^{\frac{2\alpha-1}{3\alpha-2}}.
\end{align*}
For the last inequality, we use $\frac{1}{2}\mathbb{E}_{\nu}[f^2] \leq \max\{\mathbb{E}_{\nu}[(f_-)^2],\mathbb{E}_{\nu}[(f_+)^2]\}$. Combining this with the Cauchy--Schwarz estimate~\eqref{Cauchy–Schwarz bound} gives
\begin{align*}
    \frac{1}{\mathcal{I}(\nu,\alpha)}\,
    c_4\!\left(\alpha,\tfrac{1}{\alpha},2,\mathcal{I}(\nu,\alpha)\right)
    \left(\frac{1}{2}\mathbb{E}_{\nu}[f^2]\right)^{\frac{2\alpha-1}{3\alpha-2}}
    \le 2\sqrt{\mathbb{E}_{\nu}[f^2]\mathbb{E}_{\nu}[|\nabla f|^2]}.
\end{align*}
Equivalently,
\begin{align*}
    \mathbb{E}_{\nu}[f^2]^{\frac{2\alpha-1}{3\alpha-2}-\frac{1}{2}}
    \le 2^{\frac{5\alpha-3}{3\alpha-2}}\mathcal{I}(\nu,\alpha)\,
    c_4\!\left(\alpha,\tfrac{1}{\alpha},2,\mathcal{I}(\nu,\alpha)\right)^{-1}
    \mathbb{E}_{\nu}[|\nabla f|^2]^{1/2}.
\end{align*}
Since $\tfrac{2\alpha-1}{3\alpha-2}-\tfrac{1}{2} = \tfrac{\alpha}{6\alpha-4}>0$, we finally have
\begin{align*}
    \mathbb{E}_{\nu}[f^2]
    \le 2^{\frac{10\alpha-6}{\alpha}}\mathcal{I}(\nu,\alpha)^{\frac{6\alpha-4}{\alpha}}
    c_4\!\left(\alpha,\tfrac{1}{\alpha},2,\mathcal{I}(\nu,\alpha)\right)^{-\frac{6\alpha-4}{\alpha}}
    \mathbb{E}_{\nu}[|\nabla f|^2]^{\frac{3\alpha-2}{\alpha}}.
\end{align*}
Hence,
\begin{align*}
    C_2(\alpha,\mathcal{I}(\nu,\alpha))
    &=  2^{\frac{10\alpha-6}{\alpha}}\mathcal{I}(\nu,\alpha)^{\frac{6\alpha-4}{\alpha}}
    c_4\!\left(\alpha,\tfrac{1}{\alpha},2,\mathcal{I}(\nu,\alpha)\right)^{-\frac{6\alpha-4}{\alpha}}\\
    &=  2^{\frac{16\alpha-10}{\alpha}}\mathcal{I}(\nu,\alpha)^{\frac{6\alpha-4}{\alpha}}
    \!\left[2\mathcal{I}(\nu,\alpha)^{\tfrac{2}{\alpha}}
    \!\left(\tfrac{\alpha}{1-\alpha}\right)^2
    \!\left(\tfrac{2(1-\alpha)}{3\alpha-2}\right)\!\right]^{\!\tfrac{2(1-\alpha)}{\alpha}}\!.
\end{align*}

\end{proof}

\subsection{Proof of Theorem \ref{L^1 Cheeger theorem}}
\begin{proof}
    Without loss of generality, we assume that $m_{\nu}(f)=0$. By co-area formula\cite[Lemma 3.2]{bobkov1997isoperimetric} as in the proof of Theorem \ref{L^1 theorem}, we have
    \begin{align*}
        \mathbb{E}_{\nu}[|\nabla f|] 
        &\geq \frac{1}{\mathcal{I}(\nu,\alpha)}\left(\int_{-\infty}^{0} \mathbb{P}_{\nu}(f\leq t)^{\frac{1}{\alpha}}\,dt + \int_{0}^{\infty}\mathbb{P}_{\nu}(f\geq t)^{\frac{1}{\alpha}}\,dt\right)
    \end{align*}
    By Lemma \ref{main L^2 lemma} with $\beta=\frac{1}{\alpha}$ and $\gamma=1$, we have
    \begin{align*}
        \int_{0}^{\infty}\mathbb{P}_{\nu}(f\geq t)^{\frac{1}{\alpha}}\,dt \geq c_4\left(\alpha,\frac{1}{\alpha},1,\mathcal{I}(\nu,\alpha)\right)\mathbb{E}_{\nu}[|f_+|]^{\frac{\alpha}{2\alpha-1}}
    \end{align*}
    and
    \begin{align*}
        \int_{-\infty}^{0}\mathbb{P}_{\nu}(f\leq t)^{\frac{1}{\alpha}}\,dt \geq c_4\left(\alpha,\frac{1}{\alpha},1,\mathcal{I}(\nu,\alpha)\right)\mathbb{E}_{\nu}[|f_-|]^{\frac{\alpha}{2\alpha-1}}.
    \end{align*}
    Therefore, we have
    \begin{align*}
        \mathbb{E}_{\nu}[|\nabla f|] &\geq \frac{1}{\mathcal{I}(\nu,\alpha)}c_4\left(\alpha,\frac{1}{\alpha},1,\mathcal{I}(\nu,\alpha)\right)(\mathbb{E}_{\nu}[|f_-|]^{\frac{\alpha}{2\alpha-1}}+ \mathbb{E}_{\nu}[|f_+|]^{\frac{\alpha}{2\alpha-1}})\\
        &\geq \frac{1}{\mathcal{I}(\nu,\alpha)}c_4\left(\alpha,\frac{1}{\alpha},1,\mathcal{I}(\nu,\alpha)\right)\left(\frac{1}{2}\mathbb{E}_{\nu}[|f|]\right)^{\frac{\alpha}{2\alpha-1}}.
    \end{align*}
    Again, we used $\frac{1}{2}\mathbb{E}_{\nu}[|f|] \leq \max\{\mathbb{E}_{\nu}[|f_-|],\mathbb{E}_{\nu}[|f_+|]\}$ in the last inequality. Exchanging the terms we get
    \begin{align*}
        \mathbb{E}_{\nu}[|f|] \leq 2\mathcal{I}(\nu,\alpha)^{\frac{2\alpha-1}{\alpha}}c_4\left(\alpha,\frac{1}{\alpha},1,\mathcal{I}(\nu,\alpha)\right)^{-\frac{2\alpha-1}{\alpha}}\mathbb{E}_{\nu}[|\nabla f|]^{\frac{2\alpha-1}{\alpha}}.
    \end{align*}
    Therefore, we have
    \begin{align*}
        C_3(\alpha,\mathcal{I}(\nu,\alpha)) &= 2\mathcal{I}(\nu,\alpha)^{\frac{2\alpha-1}{\alpha}}c_4\left(\alpha,\frac{1}{\alpha},1,\mathcal{I}(\nu,\alpha)\right)^{-\frac{2\alpha-1}{\alpha}}\\
        &= 2(2\mathcal{I}(\nu,\alpha))^{\frac{2\alpha-1}{\alpha}}\left[2\mathcal{I}(\nu,\alpha)^{\tfrac{1}{\alpha}}\left(\frac{\alpha}{2\alpha-1}\right) \right]^{\frac{1-\alpha}{\alpha}}.
    \end{align*}
\end{proof}

\section{Proof of Theorems \ref{product theorem} and Theorem \ref{Pareto theorem}}
In this section, we prove Theorem \ref{product theorem} and Theorem \ref{Pareto theorem} as a consequence of Theorem \ref{L^2 theorem}.
\subsection{Proof of Theorem \ref{product theorem}}
We begin by establishing a variance bound for $1$-Lipschitz functions on product measures in terms of their partial derivatives, via induction.

\begin{lemma}\label{product lemma}
    Suppose $\mathcal{C}_2(\mu,\alpha)<\infty$ on $\mathbb{R}$ for some $0<\alpha\leq1$, then for any 1-Lipschitz $f$ on $\mathbb{R}^n$, we have
    \begin{align}
        \operatorname{Var}_{\mu^n}(f) \leq \mathcal{C}_2(\nu,\alpha)\sum_{i=1}^{n} \mathbb{E}_{\mu^n}[|d_if|^2]^\alpha.
    \end{align}
\end{lemma}
\begin{proof}
    We proceed by induction on $n$.  The base case $n=1$ holds by assumption. Suppose that the inequality holds for $\mu^{n-1}$, i.e.
    \begin{align*}
        \operatorname{Var}_{\mu^{n-1}}(f) \leq \mathcal{C}_2(\nu,\alpha)\sum_{i=1}^{n-1} \mathbb{E}_{\mu^n}[|d_if|^2]^\alpha.
    \end{align*}
    We decompose $\mu^n$ as $\mu^{n-1}\times \mu^*$. Then for a point $x\in\mathbb{R}^n$, we can view it as $(x_1,...,x_{n-1}) \times (x_n)$. Then by law of total variance and applying the assumption to each term yields
    \begin{align*}
        \operatorname{Var}_{\mu^n}(f) &= \operatorname{Var}_{\mu^*}(\mathbb{E}_{\mu^{n-1}}[f]) + \mathbb{E}_{\mu^*}[\operatorname{Var}_{\mu^{n-1}}(f)]\\
        &\leq \mathcal{C}_2(\nu,\alpha)\cdot\mathbb{E}_{\mu^*}[|d_n(\mathbb{E}_{\mu^{n-1}}[f])|^2]^\alpha + \mathcal{C}_2(\nu,\alpha)\cdot \mathbb{E}_{\mu^*}\left[\sum_{i=1}^{n-1}\mathbb{E}_{\mu^{n-1}}[|d_if|^2]^\alpha\right].
    \end{align*}
    For the last inequality, we use the fact that if $f$ is 1-Lipschitz, then $\mathbb{E}_{\mu^{n-1}}[f(x_1,...,x_n)\;|\; x_n]$ is 1-Lipschitz with respect to $x_n$. And given any $x_n \in\mathbb{R}$, we have $f(x_1,...,x_{n-1}\;|\; x_n)$ is 1-Lipschitz with respect to $(x_1,...,x_{n-1})$. Note by Jensen's inequality,
    \begin{align*}
        \mathbb{E}_{\mu^*}\left[\left|d_n(\mathbb{E}_{\mu^{n-1}}[f])\right|^2\right] &= \mathbb{E}_{\mu^*}\left[\left|\mathbb{E}_{\mu^{n-1}}[d_nf]\right|^2\right]\\
        &\leq \mathbb{E}_{\mu^*}\left[\mathbb{E}_{\mu^{n-1}}\left[|d_nf|^2\right]\right]\\
        &= \mathbb{E}_{\mu^n}[|d_nf|^2]
    \end{align*}
    and since $0<\alpha<1$,
    \begin{align*}
        \mathbb{E}_{\mu^*}\left[\sum_{i=1}^{n-1}\mathbb{E}_{\mu^{n-1}}[|d_if|^2]^\alpha\right]
        &= \sum_{i=1}^{n-1}\mathbb{E}_{\mu^*}\left[\mathbb{E}_{\mu^{n-1}}[|d_if|^2]^\alpha\right]\\
        &\leq \sum_{i=1}^{n-1} \mathbb{E}_{\mu^*}\left[\mathbb{E}_{\mu^{n-1}}[|d_if|^2]\right]^{\alpha}\\
        &= \sum_{i=1}^{n-1} \mathbb{E}_{\mu^n}[|d_if|^2]^{\alpha}.
    \end{align*}
    Note that the derivative may be interchanged with expectation by the dominated convergence theorem since $f$ is $1$-Lipschitz. Combining the above estimates, we obtain
    \begin{align*}
        \operatorname{Var}_{\mu^n}(f) &\leq \mathcal{C}_2(\nu,\alpha)\cdot\mathbb{E}_{\mu^n}[|d_nf|^2]^\alpha +  \mathcal{C}_2(\nu,\alpha)\cdot \sum_{i=1}^{n-1} \mathbb{E}_{\mu^n}[|d_if|^2]^{\alpha}\\
        &= \mathcal{C}_2(\nu,\alpha)\cdot \sum_{i=1}^{n} \mathbb{E}_{\mu^n}[|d_if|^2]^{\alpha}.
    \end{align*}
\end{proof}

Now we deduce Theorem \ref{product theorem} by applying Jensen’s inequality once more.

\begin{proof}[Proof of Theorem \ref{product theorem}]
    By Lemma \ref{product lemma}, we have
    \begin{align*}
        \operatorname{Var}_{\mu^n}(f) 
        &\leq \mathcal{C}_2(\nu,\alpha)\sum_{i=1}^{n} \mathbb{E}_{\mu^n}[|d_if|^2]^\alpha \\
        & = \mathcal{C}_2(\nu,\alpha)\cdot n\cdot\sum_{i=1}^{n} \frac{1}{n}\mathbb{E}_{\nu}[|d_if|^2]^\alpha
    \end{align*}
    Since $0 < \alpha < 1$, applying Jensen’s inequality to the uniform measure on $\{1, \ldots, n\}$ gives
    \begin{align*}
        \sum_{i=1}^{n} \frac{1}{n}\mathbb{E}_{\mu^n}[|d_if|^2]^\alpha 
        &\leq \left(\sum_{i=1}^{n} \frac{1}{n}\mathbb{E}_{\mu^n}[|d_if|^2]\right)^{\alpha}\\
        &= n^{-\alpha}\left(\sum_{i=1}^{n}\mathbb{E}_{\mu^n}[|d_i f|^2]\right)^{\alpha}\\
        &= n^{-\alpha}\mathbb{E}_{\mu^n}[|\nabla f|^2]^{\alpha}
    \end{align*}
    Combining the two inequalities yields
    \begin{align*}
        \operatorname{Var}_{\mu^n}(f) \leq \mathcal{C}_2(\nu,\alpha)n^{1-\alpha}\mathbb{E}_{\mu^n}[|\nabla f|^2]^{\alpha}.
    \end{align*}
\end{proof}

\subsection{Proof of Theorem \ref{Pareto theorem}}
We first claim that for every $\lambda > 2$, the Pareto measure $\mu_{\lambda}$ possesses a finite $\alpha$-Cheeger constant $\mathcal{I}(\mu_{\lambda}, \alpha)$ for some $\alpha > \tfrac{1}{2}$.

\begin{lemma}\label{Pareto lemma}
    For any $\lambda>2$, we have
    \begin{align}
        \mathcal{I}\left(\mu_{\lambda}, \frac{\lambda-1}{\lambda}\right) \leq (\lambda-1)^{-\frac{\lambda-1}{\lambda}}
    \end{align}
\end{lemma}
\begin{proof}
    Suppose that for some measurable set $A \subseteq \mathbb{R}$, we have $\mu_{\lambda}(A) = p \leq \tfrac{1}{2}$.  
    By direct computation of the tail integral, we obtain  
    \begin{align*}
        \int_{p^{-\tfrac{1}{\lambda-1}}}^{\infty} 1 \, d\mu_{\lambda} = p.
    \end{align*}
    Hence, for any $\epsilon > 0$, the set $A$ must intersect the interval $[0, p^{-\tfrac{1}{\lambda-1}} + \epsilon)$, that is,  
    $A \cap [0, p^{-\tfrac{1}{\lambda-1}} + \epsilon) \neq \emptyset$.  
    Moreover, $\mu_{\lambda}(A) = p \leq 1-p < \mu_{\lambda}([0, p^{-\tfrac{1}{\lambda-1}} + \epsilon))$. Consequently, $\mu_{\lambda}(A^+)$ is at least the density at $p^{-\frac{1}{\lambda-1}}$:
    \begin{align*}
        \mu_{\lambda}(A^+) \geq (\lambda - 1)p^{\tfrac{\lambda}{\lambda-1}}
        = (\lambda - 1)\mu_{\lambda}(A)^{\tfrac{\lambda}{\lambda-1}}.
    \end{align*}
    Equivalently, for any measurable $A \subseteq \mathbb{R}$ such that $\mu_{\lambda}(A)\leq \frac{1}{2}$, we have
    \begin{align*}
        \mu_{\lambda}(A)
        \leq (\lambda - 1)^{-\tfrac{\lambda-1}{\lambda}}
        \mu_{\lambda}(A^+)^{\tfrac{\lambda-1}{\lambda}}.
    \end{align*}
    For $\mu_{\lambda}(A) \geq \frac{1}{2}$, the same argument works.
\end{proof}

Now we can prove Theorem \ref{Pareto theorem} by directly applying Theorem \ref{product theorem}.

\begin{proof}[Proof of Theorem \ref{Pareto theorem}]
    Note that when $\lambda>3$, we have $\alpha := \frac{\lambda-1}{\lambda} >\frac{2}{3}$ for the $\alpha$-Cheeger constant $\mathcal{I}(\mu,\alpha)$. By Theorem \ref{product theorem}, we have
    \begin{align*}
        \operatorname{Var}_{\mu_{\lambda}^n}(f) &\leq C_2\left(\frac{\lambda-1}{\lambda},(\lambda-1)^{\frac{-\lambda}{\lambda-1}}\right) n^{1-\frac{3\frac{\lambda-1}{\lambda} - 2}{\frac{\lambda-1}{\lambda}}}\mathbb{E}_{\mu^n}[|\nabla f|^2]^{\frac{3\frac{\lambda-1}{\lambda}-2}{\frac{\lambda-1}{\lambda}}}\\
        &= C_2\left(\frac{\lambda-1}{\lambda},(\lambda-1)^{\frac{-\lambda}{\lambda-1}}\right)n^{\frac{2}{\lambda-1}}\mathbb{E}_{\mu^n}[|\nabla f|^2]^{\frac{\lambda-3}{\lambda-1}}\\
        &\leq C_2\left(\frac{\lambda-1}{\lambda},(\lambda-1)^{\frac{-\lambda}{\lambda-1}}\right)n^{\frac{2}{\lambda-1}}
    \end{align*}
    Therefore, we can take
    \begin{align*}
        C(\lambda) &= C_2\left(\frac{\lambda-1}{\lambda},(\lambda-1)^{-\frac{\lambda-1}{\lambda}}\right)\\
        &= 2^{\frac{6\lambda-16}{\lambda-1}}((\lambda-1)^{-\frac{\lambda-1}{\lambda}})^{\frac{2\lambda-6}{\lambda-1}}\left[2(\lambda-1)^{-2}(\lambda-1)^2\left(\frac{2}{\lambda-3}\right)\right]^{\frac{2}{\lambda-1}}\\
        &= 2^{\frac{6\lambda-16}{\lambda-1}} (\lambda-1)^{-\frac{2\lambda-6}{\lambda}}\left(\frac{4}{\lambda-3}\right)^{\frac{2}{\lambda-1}}.
    \end{align*}
    For the lower bound, we consider the the 1-Lipschitz function $f: \mathbb{R}^n \to \mathbb{R}$ as $f(x) = |x|_{\infty} = \max\{x_1,...,x_n\}$. It is well known (see \cite[Theorem 3.3.7]{embrechts2013modelling}) that for $\lambda>3$, we have
    \begin{align*}
        \frac{1}{n^{\frac{1}{\lambda-1}}}|x|_{\infty} \xrightarrow{d} \Phi_{\lambda-1}
    \end{align*}
    as $n\to \infty$, where $\Phi_{\lambda-1}$ is the Frechet distribution. Therefore,
    \begin{align*}
        \liminf_{n\to \infty} \frac{\operatorname{Var}(|x|_{\infty})}{n^{\frac{2}{\lambda-1}}} \geq \operatorname{Var}(\Phi_{\lambda-1})> 0
    \end{align*}
    Combining the upper bound and the lower bound, we have
    \begin{align*}
        \sup_{f:1-Lipschitz} \operatorname{Var}_{\mu_{\lambda}^{n}}(f) = \Theta(n^{\frac{2}{\lambda-1}}).
    \end{align*}
\end{proof}

\section{Proof of Theorem \ref{d_p theorem}}
In this section, we prove Theorem \ref{d_p theorem}, following the approach motivated by the proofs in the previous section. One might consider this redundant, since the results presented here imply Lemma \ref{product lemma} and Theorem \ref{product theorem}. Nevertheless, we chose to first present the simpler proof in the previous section and then generalize it here, so as to provide the reader with the underlying intuition.

\begin{lemma}\label{d_p lemma}
    Suppose $\mathcal{C}_1(\mu,\alpha,2)<\infty$ for some $0<\alpha\leq 1$, then for any 1-Lipschitz $f$ on $\mathbb{R}^n$ with respect to $d_p$ for some $1<p\leq \infty$, we have for any $a\geq 1$,
    \begin{align}
        \operatorname{Var}_{\mu^n}(f) \leq \mathcal{C}_1(\mu,\alpha,2) \sum_{i=1}^{n}\mathbb{E}_{\mu^n}[|d_i f|^a]^{\frac{\alpha}{a}}.
    \end{align}
    Similarly, suppose $\mathcal{C}_2(\mu,\alpha)<\infty$ for some $0<\alpha\leq 1$, then for any 1-Lipschitz $f$ on $\mathbb{R}^n$ with respect to $d_p$ for some $1<p\leq \infty$, we have for any $a\geq 2$,
    \begin{align}
        \operatorname{Var}_{\mu^n}(f) \leq \mathcal{C}_2(\mu,\alpha) \sum_{i=1}^{n}\mathbb{E}_{\mu^n}[|d_i f|^a]^{\frac{2\alpha}{a}}.
    \end{align}
\end{lemma}
\begin{proof}
    We proceed by induction on $n$. For the base case $n=1$, we apply Jensen’s inequality to the definition of $\mathcal{C}_1(\mu,\alpha,2)$ (see Definition \eqref{alpha beta Poincaré constant}) to obtain, for any 1-Lipschitz function $f$ with respect to $d_p$,
\begin{align*}
    \operatorname{Var}_{\mu}(f)
    &\leq \mathcal{C}_1(\mu,\alpha,2)\,\mathbb{E}_{\mu}[|\nabla f|]^{\alpha}
    \leq \mathcal{C}_1(\mu,\alpha,2)\,\mathbb{E}_{\mu}[|\nabla f|^a]^{\frac{\alpha}{a}}.
\end{align*}
Now suppose the inequality holds for $\mu^{n-1}$. We decompose $\mu^n$ as $\mu^{n-1} \times \mu^*$. By the same argument as in the proof of Lemma \ref{product lemma}, we have
\begin{align*}
    \operatorname{Var}_{\mu^n}(f)
    &= \operatorname{Var}_{\mu^*}\big(\mathbb{E}_{\mu^{n-1}}[f]\big)
    + \mathbb{E}_{\mu^*}\big[\operatorname{Var}_{\mu^{n-1}}(f)\big] \\
    &\leq \mathcal{C}_1(\mu,\alpha,2)\,\mathbb{E}_{\mu^*}\big[|d_n(\mathbb{E}_{\mu^{n-1}}[f])|^a\big]^{\frac{\alpha}{a}}
    + \mathcal{C}_1(\mu,\alpha,2)\,\mathbb{E}_{\mu^*}\!\left[\sum_{i=1}^{n-1}\mathbb{E}_{\mu^{n-1}}[|d_i f|^a]^{\frac{\alpha}{a}}\right] \\
    &\leq \mathcal{C}_1(\mu,\alpha,2)\,\mathbb{E}_{\mu^n}[|d_n f|^a]^{\frac{\alpha}{a}}
    + \mathcal{C}_1(\mu,\alpha,2)\sum_{i=1}^{n-1}\mathbb{E}_{\mu^n}[|d_i f|^a]^{\frac{\alpha}{a}} \\
    &= \mathcal{C}_1(\mu,\alpha,2)\sum_{i=1}^{n}\mathbb{E}_{\mu^n}[|d_i f|^a]^{\frac{\alpha}{a}}.
\end{align*}
For the second inequality, we observe that $\frac{a}{2} \geq 1$, and therefore, by Jensen’s inequality, we have  
\begin{align*}
    \mathrm{Var}_{\mu}(f)
    \leq \mathcal{C}_2(\mu,\alpha)\mathbb{E}_{\mu}[|\nabla f|^2]^{\alpha}
    \leq \mathcal{C}_2(\mu,\alpha)\mathbb{E}_{\mu}[|\nabla f|^a]^{\frac{2\alpha}{a}}.
\end{align*}
Applying the same induction argument as before yields the second inequality.
\end{proof}

Now we can prove Theorem \ref{d_p theorem} using the same method as the previous section.
\begin{proof}[Proof of Theorem \ref{d_p theorem}]
    Taking $a = \frac{p}{p-1}$ in Lemma \ref{d_p lemma}, we have
    \begin{align*}
        \operatorname{Var}_{\mu^n}(f) &\leq \mathcal{C}_1(\mu,\alpha,2) \sum_{i=1}^n \mathbb{E}_{\mu^n}[|d_i f|^{\frac{p}{p-1}}]^{\frac{\alpha(p-1)}{p}}\\
        & = \mathcal{C}_1(\mu,\alpha,2)\cdot n \cdot \sum_{i=1}^n \frac{1}{n}\mathbb{E}_{\mu^n}[|d_i f|^{\frac{p}{p-1}}]^{\frac{\alpha(p-1)}{p}}\\
        &\leq \mathcal{C}_1(\mu,\alpha,2)\cdot n \cdot \left(\sum_{i=1}^n \frac{1}{n}\mathbb{E}_{\mu^n}[|d_i f|^{\frac{p}{p-1}}]\right)^{\frac{\alpha(p-1)}{p}}\\
        &= \mathcal{C}_1(\mu,\alpha,2)\cdot n^{1-\frac{\alpha(p-1)}{p}} \cdot \left(\sum_{i=1}^n \mathbb{E}_{\mu^n}[|d_i f|^{\frac{p}{p-1}}]\right)^{\frac{\alpha(p-1)}{p}}.
    \end{align*}

    Similarly, if $1<p\leq 2$, we have $a = \frac{p}{p-1} \geq 2$. Then by Lemma \ref{d_p lemma}, we have
    \begin{align*}
        \operatorname{Var}_{\mu^n}(f) &\leq \mathcal{C}_2(\mu,\alpha) \sum_{i=1}^{n}\mathbb{E}_{\mu^n}[|d_i f|^{\frac{p}{p-1}}]^{\frac{2\alpha(p-1)}{p}}\\
        &\leq \mathcal{C}_2(\mu,\alpha)\cdot n^{1-\frac{2\alpha(p-1)}{p}} \cdot \left(\sum_{i=1}^n \mathbb{E}_{\mu^n}[|d_i f|^{\frac{p}{p-1}}]\right)^{\frac{2\alpha(p-1)}{p}}.
    \end{align*}
\end{proof}

Note that if we take $p=2$($f$ is 1-Lipschitz with respect to the Euclidean distance), we have two ways to bound $\operatorname{Var}_{\mu^n}(f)$ if $\mathcal{I}(\mu,\alpha)<\infty$ for some $\frac{2}{3}<\alpha<1$. Theorem \ref{product theorem} tells us that
\begin{align*}
    \operatorname{Var}_{\mu^n}(f) = O(n^{\frac{2(1-\alpha)}{\alpha}}),
\end{align*}
while Theorem \ref{d_p theorem} with $\mathcal{C}_1(\mu,\alpha,2)$ obtained by Theorem \ref{L^1 theorem} tells use that
\begin{align*}
    \operatorname{Var}_{\mu^n}(f) = O(n^{\frac{4-3\alpha}{2}})
\end{align*}

Since for all $\frac{2}{3}<\alpha<1$,
\begin{align*}
    \frac{4-3\alpha}{2} \geq \frac{2(1-\alpha)}{\alpha} 
\end{align*}
the exponent $\frac{4-3\alpha}{2}$ is larger than $\frac{2(1-\alpha)}{\alpha} $. Therefore, when $p=2$ the bound obtained here has a strictly worse dependence on $n$ than the bound in Theorem \ref{product theorem}.

\section{Consequences and extensions}
\subsection{Fluctuation bound for 1-Lipschitz functions by Cheeger/$\alpha$-Cheeger constants}
Let us recall Lemma \ref{dQ lemma}. We observe that the fluctuation of any $1$-Lipschitz function is controlled by $\mathcal{I}(\nu, \alpha)$.  
We can therefore define a random variable $X = X(\alpha, \mathcal{I}(\nu, \alpha))$, representing the case of “maximal fluctuation,” and normalized so that $m(X) = 0$, by setting
\begin{align}\label{most fluctuate random variable}
    Q_X(x) :=
    \begin{cases}
        \displaystyle \int_{\frac{1}{2}}^x
        \left( \frac{\mathcal{I}(\nu, \alpha)}{1 - t} \right)^{\frac{1}{\alpha}} dt,
        & x \geq \tfrac{1}{2}, \\[1em]
        \displaystyle
        -\int_x^{\frac{1}{2}}
        \left( \frac{\mathcal{I}(\nu, \alpha)}{t} \right)^{\frac{1}{\alpha}} dt,
        & x < \tfrac{1}{2}.
    \end{cases}
\end{align}
By Lemma \ref{dQ lemma} and \ref{FTC lemma}, if $f$ is a $1$-Lipschitz function such that $Q_f(p) = Q_X(p)$ for some $p\in (0,1)$, then for any $x \geq p$,
\begin{align*}
    Q_f(x) \leq Q_X(x),
\end{align*}
and for any $y \leq p$,
\begin{align*}
    Q_f(y) \geq Q_X(y).
\end{align*}
Hence, if $\mathbb{E}_{\nu}[f] = \mathbb{E}[X] = 0$, which ensures that there exist $p$ such that $Q_f(p) = Q_X(p)$, we have the convex ordering $f \leq_{cx} X$ (see Theorem 3.A.1 in \cite{shaked2007stochastic} for the definition and related results).  It is not hard to compute the probability density function for $X(\alpha,\mathcal{I}(\nu,\alpha))$ once its quantile function if defined in \eqref{most fluctuate random variable}. For $x\in \mathbb{R}$, we have the density function for $X(\alpha,\mathcal{I}(\nu,\alpha))$ as
\begin{align*}
    f_X(x) = \frac{a(b-1)}{2}(a|x|-1)^{-b},
\end{align*}
where
\begin{align*}
    a = \frac{1-\alpha}{\alpha}\mathcal{I}(\nu,\alpha)^{-\frac{1}{\alpha}}2^{\frac{\alpha-1}{\alpha}}
\end{align*}
and
\begin{align*}
    b = \frac{1}{1-\alpha}.
\end{align*}
As a consequence, we obtain the following theorem.

\begin{theorem}
    Suppose $\mathcal{I}(\nu, \alpha) < \infty$ for some $0<\alpha<1$ and let $f$ be a $1$-Lipschitz function with respect to $\nu$ on $\mathbb{R}^n$.  Define the random variable $X = X(\alpha, \mathcal{I}(\nu, \alpha))$ as in \eqref{most fluctuate random variable}.  
    Then, if $\mathbb{E}_{\nu}(f) = 0$, we have for any convex function $\varphi: \mathbb{R} \to \mathbb{R}$,
    \begin{align}
        \mathbb{E}_{\nu}[\varphi(f)] \leq \mathbb{E}[\varphi(X)].
    \end{align}
    Moreover, for any $t>0$, we have
    \begin{align}
        \mathbb{P}_{\nu}(f-m_{\nu}(f)\geq t) \leq \mathbb{P}(X\geq t) = \frac{1}{2}(at-1)^{1-b} = O(t^{-\frac{\alpha}{1-\alpha}})
    \end{align}
    and
    \begin{align}
        \mathbb{P}_{\nu}(f-m_{\nu}(f)\leq -t) \leq \mathbb{P}(X\leq -t) = \frac{1}{2}(at-1)^{1-b} = O(t^{-\frac{\alpha}{1-\alpha}}),
    \end{align}
    where $a,b$ are defined in the equation above this theorem.
\end{theorem}

When $\mathcal{I}(\nu,1) = I(\nu) < \infty$ for the traditional Cheeger constant, the “most variable measure” takes a simpler explicit form. Define the two-sided exponential (Laplace) random variable $L_t$ by its probability density function on $\mathbb{R}$:
\begin{align}\label{two side exponential}
    l_t(x) = \frac{t}{2} e^{-t|x|}.
\end{align}
A straightforward computation yields
\begin{align*}
    \frac{d}{dx} Q_{L_t}(x) = \frac{t}{\min\{x,1-x\}}.
\end{align*}
This leads to the following theorem.

\begin{theorem}\label{Sharp bound by Cheeger theorem}
    Suppose $I(\nu)<\infty$, and let $f$ be a $1$-Lipschitz function with respect to $\nu$ on $\mathbb{R}^n$. Define $L_t$ as in \eqref{two side exponential}. If $\mathbb{E}_{\nu}(f) = 0$, then for any convex function $\varphi:\mathbb{R}\to\mathbb{R}$,
    \begin{align}
        \mathbb{E}_{\nu}[\varphi(f)] \leq \mathbb{E}[\varphi(L_{I(\nu)})].
    \end{align}
    Moreover, for any $x>0$, we have
    \begin{align}
        \mathbb{P}_{\nu}(f - m_{\nu}(f) \geq x) \leq \mathbb{P}(L_{I(\nu)} \geq x) = \frac{1}{2}e^{-I(\nu)x},
    \end{align}
    and
    \begin{align}
        \mathbb{P}_{\nu}(f - m_{\nu}(f) \leq -x) \leq \mathbb{P}(L_{I(\nu)} \leq -x) = \frac{1}{2}e^{-I(\nu)x}.
    \end{align}    
\end{theorem}

Note that applying only Cheeger’s inequality \eqref{Cheeger's inequality} yields, for any 1-Lipschitz function $f$,
\begin{align*}
    \mathrm{Var}_{\nu}(f) \leq C_P(\nu)\mathbb{E}_{\nu}[|\nabla f|^2] \leq C_P(\nu) \leq 4I(\nu)^2.
\end{align*}
On the other hand, applying Theorem \ref{Sharp bound by Cheeger theorem} with $\varphi(x)=x^2$ gives
\begin{align*}
    \mathrm{Var}_{\nu}(f) \leq \mathbb{E}[L_{I(\nu)}^2] = 2I(\nu)^2,
\end{align*}
which provides a strictly sharper bound. Another point to note is that if we regard $L_t$ as the 1-Lipschitz function $f(x) = x$ under the measure $\mu_t$ defined by the density $l_t$ on $\mathbb{R}$, then $\mu_t$ is a log-concave measure on $\mathbb{R}$. Moreover, by \cite{bobkov1996extremal}, we have $I(\mu_t) = t$. Hence, the bounds in Theorem \ref{Sharp bound by Cheeger theorem} are attained by this construction on $\mathbb{R}$.

\subsection{Isoperimetric inequalities for product measures}
It is well known that concentration inequalities for 1-Lipschitz functions imply corresponding isoperimetric inequalities. In \cite{bobkov1996variance}, the authors established several isoperimetric inequalities by employing variance bounds for 1-Lipschitz functions (with respect to a different metric). The same approach can be applied in our setting. Nevertheless, we choose to present a slightly weaker isoperimetric inequality using a more direct and elementary argument.

\begin{theorem}
    If $\mathcal{C}_2(\mu,\alpha)<\infty$ for some $0<\alpha<1$, then for any sets $A,B \subseteq \mathbb{R}^n$, we have
    \begin{align}
        d(A,B) \leq n^{\frac{1-\alpha}{2}}\sqrt{\frac{\mathcal{C}_2(\mu,\alpha)}{\mu^n(A)\mu^n(B)}}.
    \end{align}
    Together with Theorem \ref{L^2 theorem}, if $\mathcal{I}(\mu,\alpha)<\infty$ for some $\frac{2}{3}<\alpha<1$, then
    \begin{align}
        d(A,B) \leq n^{\frac{1-\alpha}{\alpha}}\sqrt{\frac{C_2(\alpha,\mathcal{I}(\mu,\alpha))}{\mu^n(A)\mu^n(B)}}.
    \end{align}
\end{theorem}
\begin{proof}
    Consider the 1-Lipschitz function $f: \mathbb{R}^n \to \mathbb{R}$ defined by
    \begin{align*}
        f(x) = \min\bigl\{d(x,A),\, d(A,B)\bigr\}.
    \end{align*}
    Let $x$ and $y$ be independent random points distributed according to $\mu^n$. Then
    \begin{align*}
        \mathrm{Var}_{\mu^n}(f) &= \frac{1}{2}\mathbb{E}\bigl[(f(x) - f(y))^2\bigr] \\
        &\geq \frac{1}{2} \bigl(\mathbb{P}(x \in A, y \in B) + \mathbb{P}(x \in B, y \in A)\bigr) d(A,B)^2 \\
        &= \mu^n(A)\, \mu^n(B)\, d(A,B)^2.
    \end{align*}
    For the second inequality, we use the fact that $(f(x)-f(y))^2\geq 0$ and when $x\in A, y\in B$ or $x\in B,y\in A$, we have $(f(x)-f(y))^2= d(A,B)^2$.
    On the other hand, by Theorem \ref{product theorem}, we have
    \begin{align*}
        \mathrm{Var}_{\mu^n}(f) \leq \mathcal{C}_2(\mu,\alpha)\, n^{1-\alpha}.
    \end{align*}
    Combining these two inequalities yields
    \begin{align*}
        d(A,B) \leq n^{\frac{1-\alpha}{2}} \sqrt{\frac{\mathcal{C}_2(\mu,\alpha)}{\mu^n(A)\, \mu^n(B)}}.
    \end{align*}

\end{proof}

\subsection{1-Lipschitz function with respect to $d_p (1<p\leq2)$ by traditional Poincaré constant}
Suppose that $C_P(\mu)<\infty$ for the classical Poincaré constant. It is well known that for any Lipschitz function $f$,
\begin{align*}
    \mathrm{Var}_{\mu^n}(f) \leq C_P(\mu)\, \mathbb{E}_{\mu^n}[|\nabla f|^2].
\end{align*}
Consequently, for any $1$-Lipschitz function $f$ with respect to $d_2$, we have $\mathrm{Var}_{\mu^n}(f) \leq C_P(\mu) = O(1)$.  
On the other hand, if we consider the $1$-Lipschitz function (with respect to $d_1$) defined by $f(x) = \sum_{i=1}^n x_i$, then $\mathrm{Var}_{\mu^n}(f) = O(n)$.  

Therefore, it is natural to ask about the growth rate of $\mathrm{Var}_{\mu^n}(f)$ when $f$ is a $1$-Lipschitz function with respect to $d_p$ for $1 < p \leq 2$.  
This question can be answered precisely, up to the correct asymptotic order, under the assumption that $C_P(\mu) < \infty$.

\begin{theorem}\label{sharp Poincaré d_p theorem}
    Suppose $C_P(\mu)<\infty$, then for any 1-Lipschitz $f$ with respect to $d_p$ for some $1<p\leq 2$, we have
    \begin{align}
        \operatorname{Var}_{\mu^n}(f) \leq C_P(\mu)\cdot n^{\frac{2-p}{p}} \cdot \left(\sum_{i=1}^{n}\mathbb{E}_{\mu^n}[|d_i f|^{\frac{p}{p-1}}]\right)^{\frac{2(p-1)}{p}} = O(n^{\frac{2-p}{p}})
    \end{align}
    This bound is asymptotically tight.
\end{theorem}
\begin{proof}
    Since $\frac{p}{p-1}\geq 2$, by Jensen's inequality we have
    \begin{align*}
        \operatorname{Var}_{\mu}(f) &\leq C_P(\mu) \mathbb{E}_{\mu}[|\nabla f|^2] \leq C_P(\mu) \mathbb{E}_{\mu}[|\nabla f|^{\frac{p}{p-1}}]^{\frac{2(p-1)}{p}}.
    \end{align*}
    Then by the same argument as in proof of Lemma \ref{d_p lemma} and Theorem \ref{d_p theorem},
    \begin{align*}
        \operatorname{Var}_{\mu^n}(f) &\leq C_P(\mu) \sum_{i=1}^n \mathbb{E}_{\mu^n}[|d_i f|^{\frac{p}{p-1}}]^{\frac{2(p-1)}{p}}\\
        &\leq C_P(\mu) \cdot n^{1-\frac{2(p-1)}{p}} \cdot \left(\sum_{i=1}^{n}\mathbb{E}_{\mu^n}[|d_i f|^{\frac{p}{p-1}}]\right)^{\frac{2(p-1)}{p}}.
    \end{align*}
    Now take the 1-Lipschitz function with respect to $d_p$ defined by
    \begin{align*}
        f(x) = n^{-\frac{p-1}{p}}\sum_{i=1}^{n}x_i.
    \end{align*}
    We have
    \begin{align*}
        \operatorname{Var}_{\mu^n}(f) = n^{1-\frac{2(p-1)}{p}} \operatorname{Var}_{\mu}(x) = \Theta(n^{\frac{2-p}{p}}).
    \end{align*}
\end{proof}

\subsection{Applications to variance of eigenvalues for heavy tail random matrices}
Consider a symmetric random matrix $A = \{A_{i,j}\}_{i,j \leq n}$ such that $A_{i,j} = A_{j,i}$ and, for $i \geq j$, the entries $A_{i,j}$ are i.i.d.\ random variables following the Pareto distribution $\mu_{\lambda}$ for some $\lambda > 3$.  
The same argument applies to other heavy-tailed distributions, but we restrict attention to the Pareto family for simplicity.  
Let $A$ have a decreasing sequence of eigenvalues $\lambda_1, \dots, \lambda_n$.  
By the Hoffman–Wielandt inequality~\cite[Theorem~VI.4.1]{bhatia2013matrix}, if $B$ is another $n \times n$ matrix with decreasing eigenvalues $\lambda_1', \dots, \lambda_n'$, then  
\begin{align*}
    \sum_{i=1}^n |\lambda_i - \lambda_i'|^2
    \leq \sum_{i,j \leq n} |A_{i,j} - B_{i,j}|^2.
\end{align*}
Now, for any $i \in \{1, \dots, n\}$, define the function $f_i : \mathbb{R}^{\frac{n(n+1)}{2}} \to \mathbb{R}$ by  
\begin{align*}
    f_i(\{A_{i,j}\}_{j \leq i \leq n})
    = \frac{1}{\sqrt{2}} \lambda_i.
\end{align*}
Then $f_i$ is a $1$-Lipschitz function.  
We can therefore state the following theorem.

\begin{theorem}\label{random matrix theorem}
    Suppose $A$ is a symmetric random matrix with i.i.d.\ entries on the upper triangle distributed according to $\mu_{\lambda}$ for $\lambda > 3$.  
    Then,  for all $1 \leq i \leq n$,
    \begin{align}
        \operatorname{Var}(\lambda_i)
        \leq 2C(\lambda)\left(\frac{n(n-1)}{2}\right)^{\frac{2}{\lambda-1}},
    \end{align}
    where $C(\lambda)$ is the constant defined in Theorem~\ref{Pareto theorem}.
\end{theorem}

It is well known that when $1 < \lambda < 5$,  
\begin{align*}
    \left(\frac{n(n-1)}{2}\right)^{-\frac{1}{\lambda-1}} \lambda_1
    \xrightarrow{d} \Phi_{\lambda-1},
\end{align*}
as $n \to \infty$, where $\Phi_{\lambda-1}$ denotes the Fréchet distribution~\cite[Corollary~1]{auffinger2009poisson}.  
However, to the author’s knowledge, a variance bound has not been established in this regime.  
Theorem~\ref{random matrix theorem} confirms that when $3 < \lambda < 5$, the variance can be bounded on the correct order of scaling.  
For $\lambda > 5$, it is known that the largest eigenvalue $\lambda_1$ exhibits Tracy–Widom fluctuations~\cite{biroli2007top}.  
However, even in this regime, an explicit variance bound does not seem to be available in the literature.  
Theorem~\ref{random matrix theorem} partially fills this gap, although the resulting variance bound may not be optimal.  
By contrast, when the matrix entries are Gaussian random variables or have exponential tails, variance and fluctuation bounds for eigenvalues are well known; see, for example, \cite{ledoux2007deviation} and \cite{dallaporta2012eigenvalue}.

\subsection{Tightness of Theorems \ref{L^2 theorem} and \ref{L^1 Cheeger theorem}}
In this section, we claim that Theorems~\ref{L^2 theorem} and~\ref{L^1 Cheeger theorem} are optimal up to constant scaling.  
Specifically, we can construct a measure $\mu$ such that $\mathcal{I}(\mu,\alpha) < \infty$, while  
$\mathcal{C}_2(\mu,\alpha_1) = \infty$ for all $\alpha_1 > \frac{3\alpha - 2}{\alpha}$ and  
$\mathcal{C}_1(\mu,\alpha_2,1) = \infty$ for all $\alpha_2 > \frac{2\alpha - 1}{\alpha}$.  

We consider the Pareto probability measure $\mu_{\frac{1}{1 - \alpha}}$.  
By Lemma~\ref{Pareto lemma}, we have  
\begin{align*}
    \mathcal{I}\left(\mu_{\frac{1}{1 - \alpha}}, \alpha\right) < \infty.
\end{align*}

Of course, one might attempt to argue that if there exists $\alpha_1 > \frac{3\alpha - 2}{\alpha}$ such that $\mathcal{C}_{2}(\mu,\alpha_1) < \infty$, then by Theorem~\ref{product theorem} we would have  
$\mathrm{Var}_{\mu^n_{\lambda}}(|x|_{\infty}) = o(n^{\frac{2}{\lambda - 1}})$,  
which would lead to a contradiction.  
However, we choose to argue more directly.  
Observe that conditioning on the coordinates $x_1,...,x_{n-1}$, the function $|x|_{\infty}$ behaves as an activation function along the $x_n$ direction with a constant shift:  
$\frac{d}{dx_n}|x|_{\infty} = 0$ for $x_n \in [-\infty, m)$ and $\frac{d}{dx_n}|x|_{\infty} = 1$ for $x_n \in (m, \infty)$, for some $m \geq 1$.  
Thus, in one dimension, we expect this activation-type behavior to generate the desired contradiction.  
For $m \geq 1$, consider the $1$-Lipschitz activation functions $f_m : \mathbb{R} \to \mathbb{R}$ defined by  
\begin{align*}
    f_m(x) = \mathbb{I}(x \geq m) \cdot (x - m).
\end{align*}
It is straightforward to verify that  
\begin{align*}
    \mathrm{Var}_{\mu_{\frac{1}{1 - \alpha}}}(f_m) 
    &= \mathbb{E}_{\mu_{\frac{1}{1 - \alpha}}}[f_m^2] - \mathbb{E}_{\mu_{\frac{1}{1 - \alpha}}}[f_m]^2\\
    &= \int_{m}^{\infty} (x - m)^2 \cdot \frac{\alpha}{1 - \alpha} \cdot x^{-\frac{1}{1 - \alpha}} \, dx  - \left(\int_{m}^{\infty} (x - m) \cdot \frac{\alpha}{1 - \alpha} \cdot x^{-\frac{1}{1 - \alpha}} \, dx\right)^2\\
    &= \Theta\!\left(m^{-\frac{3\alpha - 2}{1 - \alpha}}\right) - \Theta\left(m^{-\frac{2(2\alpha-1)}{1-\alpha}}\right)\\
    &= \Theta\!\left(m^{-\frac{3\alpha - 2}{1 - \alpha}}\right)
\end{align*}
and
\begin{align*}
    \mathbb{E}_{\mu_{\frac{1}{1 - \alpha}}}[|f_m'|^2] 
    = \int_{m}^{\infty} \frac{\alpha}{1 - \alpha} \cdot x^{-\frac{1}{1 - \alpha}} \, dx 
    = \Theta\!\left(m^{-\frac{\alpha}{1 - \alpha}}\right).
\end{align*}
Therefore,
\begin{align*}
    \mathcal{C}_2(\mu_{\frac{1}{1 - \alpha}}, \alpha_1)
    \geq \limsup_{m \to \infty}
    \frac{\mathrm{Var}_{\mu_{\frac{1}{1 - \alpha}}}(f_m)}
         {\mathbb{E}_{\mu_{\frac{1}{1 - \alpha}}}[|f_m'|^2]^{\alpha_1}}
    = \limsup_{m\to \infty}\Theta\!\left(m^{-\frac{3\alpha - 2}{1 - \alpha}
       + \frac{\alpha}{1 - \alpha} \cdot \alpha_1}\right)
    = \infty.
\end{align*}

Similarly, when $m$ is large enough such that $\mathbb{P}_{\mu_{\frac{1}{1 - \alpha}}}(f_m = 0) \geq \frac{1}{2}$ we have  
\begin{align*}
    \mathbb{E}_{\mu_{\frac{1}{1 - \alpha}}}[|f_m - m_{\mu_{\frac{1}{1 - \alpha}}}(f_m)|]
    &= \mathbb{E}_{\mu_{\frac{1}{1 - \alpha}}}[|f_m|]
    = \int_{m}^{\infty} (x - m) \cdot \frac{\alpha}{1 - \alpha} \cdot x^{-\frac{1}{1 - \alpha}} \, dx 
    = \Theta\!\left(m^{-\frac{2\alpha - 1}{1 - \alpha}}\right),
\end{align*}
and
\begin{align*}
    \mathbb{E}_{\mu_{\frac{1}{1 - \alpha}}}[|f_m'|]
    = \int_{m}^{\infty} \frac{\alpha}{1 - \alpha} \cdot x^{-\frac{1}{1 - \alpha}} \, dx 
    = \Theta\!\left(m^{-\frac{\alpha}{1 - \alpha}}\right).
\end{align*}
Therefore,
\begin{align*}
    \mathcal{C}_1(\mu_{\frac{1}{1 - \alpha}}, \alpha_2, 1)
    \geq \limsup_{m \to \infty}
    \frac{\mathbb{E}_{\mu_{\frac{1}{1 - \alpha}}}[|f_m - m_{\mu_{\frac{1}{1 - \alpha}}}(f_m)|]}
         {\mathbb{E}_{\mu_{\frac{1}{1 - \alpha}}}[|f_m'|]^{\alpha_2}}
    = \limsup_{m\to \infty}\Theta\!\left(m^{-\frac{2\alpha - 1}{1 - \alpha}
       + \frac{\alpha}{1 - \alpha} \cdot \alpha_2}\right)
    = \infty.
\end{align*}

\section*{Acknowledgements}
I am grateful to my advisor, Philippe Sosoe, for his detailed comments and insightful guidance throughout this work.  
I also thank Kengo Kato for his valuable feedback and helpful discussions related to this paper.

\bibliographystyle{apacite}
\bibliography{References}

\begin{thebibliography}{}

\bibitem [\protect \citeauthoryear {%
Auffinger%
, Ben~Arous%
\BCBL {}\ \BBA {} P{\'e}ch{\'e}%
}{%
Auffinger%
\ \protect \BOthers {.}}{%
{\protect \APACyear {2009}}%
}]{%
auffinger2009poisson}
\APACinsertmetastar {%
auffinger2009poisson}%
\begin{APACrefauthors}%
Auffinger, A.%
, Ben~Arous, G.%
\BCBL {}\ \BBA {} P{\'e}ch{\'e}, S.%
\end{APACrefauthors}%
\unskip\
\newblock
\APACrefYearMonthDay{2009}{}{}.
\newblock
{\BBOQ}\APACrefatitle {Poisson convergence for the largest eigenvalues of heavy tailed random matrices} {Poisson convergence for the largest eigenvalues of heavy tailed random matrices}.{\BBCQ}
\newblock
\BIn{} \APACrefbtitle {Annales de l'IHP Probabilit{\'e}s et statistiques} {Annales de l'ihp probabilit{\'e}s et statistiques}\ (\BVOL~45, \BPGS\ 589--610).
\PrintBackRefs{\CurrentBib}

\bibitem [\protect \citeauthoryear {%
Bakry%
, Gentil%
\BCBL {}\ \BBA {} Ledoux%
}{%
Bakry%
\ \protect \BOthers {.}}{%
{\protect \APACyear {2013}}%
}]{%
bakry2013analysis}
\APACinsertmetastar {%
bakry2013analysis}%
\begin{APACrefauthors}%
Bakry, D.%
, Gentil, I.%
\BCBL {}\ \BBA {} Ledoux, M.%
\end{APACrefauthors}%
\unskip\
\newblock
\APACrefYear{2013}.
\newblock
\APACrefbtitle {Analysis and geometry of Markov diffusion operators} {Analysis and geometry of markov diffusion operators}\ (\BVOL~348).
\newblock
\APACaddressPublisher{}{Springer Science \& Business Media}.
\PrintBackRefs{\CurrentBib}

\bibitem [\protect \citeauthoryear {%
Barthe%
, Cattiaux%
\BCBL {}\ \BBA {} Roberto%
}{%
Barthe%
\ \protect \BOthers {.}}{%
{\protect \APACyear {2005}}%
}]{%
barthe2005concentration}
\APACinsertmetastar {%
barthe2005concentration}%
\begin{APACrefauthors}%
Barthe, F.%
, Cattiaux, P.%
\BCBL {}\ \BBA {} Roberto, C.%
\end{APACrefauthors}%
\unskip\
\newblock
\APACrefYearMonthDay{2005}{}{}.
\newblock
{\BBOQ}\APACrefatitle {Concentration for independent random variables with heavy tails} {Concentration for independent random variables with heavy tails}.{\BBCQ}
\newblock
\APACjournalVolNumPages{Applied Mathematics Research eXpress}{2005}{2}{39--60}.
\PrintBackRefs{\CurrentBib}

\bibitem [\protect \citeauthoryear {%
Bhatia%
}{%
Bhatia%
}{%
{\protect \APACyear {2013}}%
}]{%
bhatia2013matrix}
\APACinsertmetastar {%
bhatia2013matrix}%
\begin{APACrefauthors}%
Bhatia, R.%
\end{APACrefauthors}%
\unskip\
\newblock
\APACrefYear{2013}.
\newblock
\APACrefbtitle {Matrix analysis} {Matrix analysis}\ (\BVOL~169).
\newblock
\APACaddressPublisher{}{Springer Science \& Business Media}.
\PrintBackRefs{\CurrentBib}

\bibitem [\protect \citeauthoryear {%
Biroli%
, Bouchaud%
\BCBL {}\ \BBA {} Potters%
}{%
Biroli%
\ \protect \BOthers {.}}{%
{\protect \APACyear {2007}}%
}]{%
biroli2007top}
\APACinsertmetastar {%
biroli2007top}%
\begin{APACrefauthors}%
Biroli, G.%
, Bouchaud, J\BHBI P.%
\BCBL {}\ \BBA {} Potters, M.%
\end{APACrefauthors}%
\unskip\
\newblock
\APACrefYearMonthDay{2007}{}{}.
\newblock
{\BBOQ}\APACrefatitle {On the top eigenvalue of heavy-tailed random matrices} {On the top eigenvalue of heavy-tailed random matrices}.{\BBCQ}
\newblock
\APACjournalVolNumPages{Europhysics Letters}{78}{1}{10001}.
\PrintBackRefs{\CurrentBib}

\bibitem [\protect \citeauthoryear {%
S.~Bobkov%
}{%
S.~Bobkov%
}{%
{\protect \APACyear {1996}}%
}]{%
bobkov1996extremal}
\APACinsertmetastar {%
bobkov1996extremal}%
\begin{APACrefauthors}%
Bobkov, S.%
\end{APACrefauthors}%
\unskip\
\newblock
\APACrefYearMonthDay{1996}{}{}.
\newblock
{\BBOQ}\APACrefatitle {Extremal properties of half-spaces for log-concave distributions} {Extremal properties of half-spaces for log-concave distributions}.{\BBCQ}
\newblock
\APACjournalVolNumPages{The Annals of Probability}{24}{1}{35--48}.
\PrintBackRefs{\CurrentBib}

\bibitem [\protect \citeauthoryear {%
S.~Bobkov%
}{%
S.~Bobkov%
}{%
{\protect \APACyear {2007}}%
}]{%
bobkov2007large}
\APACinsertmetastar {%
bobkov2007large}%
\begin{APACrefauthors}%
Bobkov, S.%
\end{APACrefauthors}%
\unskip\
\newblock
\APACrefYearMonthDay{2007}{}{}.
\newblock
{\BBOQ}\APACrefatitle {Large deviations and isoperimetry over convex probability measures with heavy tails} {Large deviations and isoperimetry over convex probability measures with heavy tails}.{\BBCQ}
\newblock

\PrintBackRefs{\CurrentBib}

\bibitem [\protect \citeauthoryear {%
S.~Bobkov%
\ \BBA {} Zegarlinski%
}{%
S.~Bobkov%
\ \BBA {} Zegarlinski%
}{%
{\protect \APACyear {2009}}%
}]{%
bobkov2009distributions}
\APACinsertmetastar {%
bobkov2009distributions}%
\begin{APACrefauthors}%
Bobkov, S.%
\BCBT {}\ \BBA {} Zegarlinski, B.%
\end{APACrefauthors}%
\unskip\
\newblock
\APACrefYearMonthDay{2009}{}{}.
\newblock
{\BBOQ}\APACrefatitle {Distributions with slow tails and ergodicity of Markov semigroups in infinite dimensions} {Distributions with slow tails and ergodicity of markov semigroups in infinite dimensions}.{\BBCQ}
\newblock
\BIn{} \APACrefbtitle {Around the Research of Vladimir Maz'ya I: Function Spaces} {Around the research of vladimir maz'ya i: Function spaces}\ (\BPGS\ 13--79).
\newblock
\APACaddressPublisher{}{Springer}.
\PrintBackRefs{\CurrentBib}

\bibitem [\protect \citeauthoryear {%
S\BPBI G.~Bobkov%
\ \BBA {} Houdr{\'e}%
}{%
S\BPBI G.~Bobkov%
\ \BBA {} Houdr{\'e}%
}{%
{\protect \APACyear {1996}}%
}]{%
bobkov1996variance}
\APACinsertmetastar {%
bobkov1996variance}%
\begin{APACrefauthors}%
Bobkov, S\BPBI G.%
\BCBT {}\ \BBA {} Houdr{\'e}, C.%
\end{APACrefauthors}%
\unskip\
\newblock
\APACrefYearMonthDay{1996}{}{}.
\newblock
{\BBOQ}\APACrefatitle {Variance of Lipschitz functions and an isoperimetric problem for a class of product measures} {Variance of lipschitz functions and an isoperimetric problem for a class of product measures}.{\BBCQ}
\newblock
\APACjournalVolNumPages{Bernoulli}{}{}{249--255}.
\PrintBackRefs{\CurrentBib}

\bibitem [\protect \citeauthoryear {%
S\BPBI G.~Bobkov%
\ \BBA {} Houdr{\'e}%
}{%
S\BPBI G.~Bobkov%
\ \BBA {} Houdr{\'e}%
}{%
{\protect \APACyear {1997}}%
}]{%
bobkov1997isoperimetric}
\APACinsertmetastar {%
bobkov1997isoperimetric}%
\begin{APACrefauthors}%
Bobkov, S\BPBI G.%
\BCBT {}\ \BBA {} Houdr{\'e}, C.%
\end{APACrefauthors}%
\unskip\
\newblock
\APACrefYearMonthDay{1997}{}{}.
\newblock
{\BBOQ}\APACrefatitle {Isoperimetric constants for product probability measures} {Isoperimetric constants for product probability measures}.{\BBCQ}
\newblock
\APACjournalVolNumPages{The Annals of Probability}{}{}{184--205}.
\PrintBackRefs{\CurrentBib}

\bibitem [\protect \citeauthoryear {%
S\BPBI G.~Bobkov%
\ \BBA {} Ledoux%
}{%
S\BPBI G.~Bobkov%
\ \BBA {} Ledoux%
}{%
{\protect \APACyear {2009}}%
}]{%
bobkov2009weighted}
\APACinsertmetastar {%
bobkov2009weighted}%
\begin{APACrefauthors}%
Bobkov, S\BPBI G.%
\BCBT {}\ \BBA {} Ledoux, M.%
\end{APACrefauthors}%
\unskip\
\newblock
\APACrefYearMonthDay{2009}{}{}.
\newblock
{\BBOQ}\APACrefatitle {Weighted Poincar{\'e}-type inequalities for Cauchy and other convex measures} {Weighted poincar{\'e}-type inequalities for cauchy and other convex measures}.{\BBCQ}
\newblock

\PrintBackRefs{\CurrentBib}

\bibitem [\protect \citeauthoryear {%
Boucheron%
, Lugosi%
\BCBL {}\ \BBA {} Massart%
}{%
Boucheron%
\ \protect \BOthers {.}}{%
{\protect \APACyear {2013}}%
}]{%
boucheron2013concentration}
\APACinsertmetastar {%
boucheron2013concentration}%
\begin{APACrefauthors}%
Boucheron, S.%
, Lugosi, G.%
\BCBL {}\ \BBA {} Massart, P.%
\end{APACrefauthors}%
\unskip\
\newblock
\APACrefYear{2013}.
\newblock
\APACrefbtitle {\textit{Concentration Inequalities: A Nonasymptotic Theory of Independence}} {\textit{Concentration Inequalities: A Nonasymptotic Theory of Independence}}.
\newblock
\APACaddressPublisher{}{Oxford Univ. Press}.
\PrintBackRefs{\CurrentBib}

\bibitem [\protect \citeauthoryear {%
Cattiaux%
, Gozlan%
, Guillin%
\BCBL {}\ \BBA {} Roberto%
}{%
Cattiaux%
\ \protect \BOthers {.}}{%
{\protect \APACyear {2010}}%
}]{%
cattiaux2010functional}
\APACinsertmetastar {%
cattiaux2010functional}%
\begin{APACrefauthors}%
Cattiaux, P.%
, Gozlan, N.%
, Guillin, A.%
\BCBL {}\ \BBA {} Roberto, C.%
\end{APACrefauthors}%
\unskip\
\newblock
\APACrefYearMonthDay{2010}{}{}.
\newblock
{\BBOQ}\APACrefatitle {Functional inequalities for heavy tailed distributions and application to isoperimetry} {Functional inequalities for heavy tailed distributions and application to isoperimetry}.{\BBCQ}
\newblock

\PrintBackRefs{\CurrentBib}

\bibitem [\protect \citeauthoryear {%
Chatterjee%
}{%
Chatterjee%
}{%
{\protect \APACyear {2014}}%
}]{%
chatterjee2014superconcentration}
\APACinsertmetastar {%
chatterjee2014superconcentration}%
\begin{APACrefauthors}%
Chatterjee, S.%
\end{APACrefauthors}%
\unskip\
\newblock
\APACrefYear{2014}.
\newblock
\APACrefbtitle {Superconcentration and related topics} {Superconcentration and related topics}\ (\BVOL~15).
\newblock
\APACaddressPublisher{}{Springer}.
\PrintBackRefs{\CurrentBib}

\bibitem [\protect \citeauthoryear {%
Dallaporta%
}{%
Dallaporta%
}{%
{\protect \APACyear {2012}}%
}]{%
dallaporta2012eigenvalue}
\APACinsertmetastar {%
dallaporta2012eigenvalue}%
\begin{APACrefauthors}%
Dallaporta, S.%
\end{APACrefauthors}%
\unskip\
\newblock
\APACrefYearMonthDay{2012}{}{}.
\newblock
{\BBOQ}\APACrefatitle {Eigenvalue variance bounds for Wigner and covariance random matrices} {Eigenvalue variance bounds for wigner and covariance random matrices}.{\BBCQ}
\newblock
\APACjournalVolNumPages{Random Matrices: Theory and Applications}{1}{03}{1250007}.
\PrintBackRefs{\CurrentBib}

\bibitem [\protect \citeauthoryear {%
Efron%
\ \BBA {} Stein%
}{%
Efron%
\ \BBA {} Stein%
}{%
{\protect \APACyear {1981}}%
}]{%
efron1981jackknife}
\APACinsertmetastar {%
efron1981jackknife}%
\begin{APACrefauthors}%
Efron, B.%
\BCBT {}\ \BBA {} Stein, C.%
\end{APACrefauthors}%
\unskip\
\newblock
\APACrefYearMonthDay{1981}{}{}.
\newblock
{\BBOQ}\APACrefatitle {The jackknife estimate of variance} {The jackknife estimate of variance}.{\BBCQ}
\newblock
\APACjournalVolNumPages{The Annals of Statistics}{}{}{586--596}.
\PrintBackRefs{\CurrentBib}

\bibitem [\protect \citeauthoryear {%
Embrechts%
, Kl{\"u}ppelberg%
\BCBL {}\ \BBA {} Mikosch%
}{%
Embrechts%
\ \protect \BOthers {.}}{%
{\protect \APACyear {2013}}%
}]{%
embrechts2013modelling}
\APACinsertmetastar {%
embrechts2013modelling}%
\begin{APACrefauthors}%
Embrechts, P.%
, Kl{\"u}ppelberg, C.%
\BCBL {}\ \BBA {} Mikosch, T.%
\end{APACrefauthors}%
\unskip\
\newblock
\APACrefYear{2013}.
\newblock
\APACrefbtitle {Modelling extremal events: for insurance and finance} {Modelling extremal events: for insurance and finance}\ (\BVOL~33).
\newblock
\APACaddressPublisher{}{Springer Science \& Business Media}.
\PrintBackRefs{\CurrentBib}

\bibitem [\protect \citeauthoryear {%
Klartag%
\ \BBA {} Lehec%
}{%
Klartag%
\ \BBA {} Lehec%
}{%
{\protect \APACyear {2024}}%
}]{%
klartag2024isoperimetric}
\APACinsertmetastar {%
klartag2024isoperimetric}%
\begin{APACrefauthors}%
Klartag, B.%
\BCBT {}\ \BBA {} Lehec, J.%
\end{APACrefauthors}%
\unskip\
\newblock
\APACrefYearMonthDay{2024}{}{}.
\newblock
{\BBOQ}\APACrefatitle {Isoperimetric inequalities in high-dimensional convex sets} {Isoperimetric inequalities in high-dimensional convex sets}.{\BBCQ}
\newblock
\APACjournalVolNumPages{arXiv preprint arXiv:2406.01324}{}{}{}.
\PrintBackRefs{\CurrentBib}

\bibitem [\protect \citeauthoryear {%
Ledoux%
}{%
Ledoux%
}{%
{\protect \APACyear {2007}}%
}]{%
ledoux2007deviation}
\APACinsertmetastar {%
ledoux2007deviation}%
\begin{APACrefauthors}%
Ledoux, M.%
\end{APACrefauthors}%
\unskip\
\newblock
\APACrefYearMonthDay{2007}{}{}.
\newblock
{\BBOQ}\APACrefatitle {Deviation inequalities on largest eigenvalues} {Deviation inequalities on largest eigenvalues}.{\BBCQ}
\newblock
\BIn{} \APACrefbtitle {Geometric Aspects of Functional Analysis: Israel Seminar 2004--2005} {Geometric aspects of functional analysis: Israel seminar 2004--2005}\ (\BPGS\ 167--219).
\PrintBackRefs{\CurrentBib}

\bibitem [\protect \citeauthoryear {%
O'Donnell%
}{%
O'Donnell%
}{%
{\protect \APACyear {2014}}%
}]{%
o2014analysis}
\APACinsertmetastar {%
o2014analysis}%
\begin{APACrefauthors}%
O'Donnell, R.%
\end{APACrefauthors}%
\unskip\
\newblock
\APACrefYear{2014}.
\newblock
\APACrefbtitle {Analysis of boolean functions} {Analysis of boolean functions}.
\newblock
\APACaddressPublisher{}{Cambridge University Press}.
\PrintBackRefs{\CurrentBib}

\bibitem [\protect \citeauthoryear {%
R{\"o}ckner%
\ \BBA {} Wang%
}{%
R{\"o}ckner%
\ \BBA {} Wang%
}{%
{\protect \APACyear {2001}}%
}]{%
rockner2001weak}
\APACinsertmetastar {%
rockner2001weak}%
\begin{APACrefauthors}%
R{\"o}ckner, M.%
\BCBT {}\ \BBA {} Wang, F\BHBI Y.%
\end{APACrefauthors}%
\unskip\
\newblock
\APACrefYearMonthDay{2001}{}{}.
\newblock
{\BBOQ}\APACrefatitle {Weak Poincar{\'e} inequalities and L2-convergence rates of Markov semigroups} {Weak poincar{\'e} inequalities and l2-convergence rates of markov semigroups}.{\BBCQ}
\newblock
\APACjournalVolNumPages{Journal of Functional Analysis}{185}{2}{564--603}.
\PrintBackRefs{\CurrentBib}

\bibitem [\protect \citeauthoryear {%
Royden%
\ \BBA {} Fitzpatrick%
}{%
Royden%
\ \BBA {} Fitzpatrick%
}{%
{\protect \APACyear {1988}}%
}]{%
royden1988real}
\APACinsertmetastar {%
royden1988real}%
\begin{APACrefauthors}%
Royden, H\BPBI L.%
\BCBT {}\ \BBA {} Fitzpatrick, P.%
\end{APACrefauthors}%
\unskip\
\newblock
\APACrefYear{1988}.
\newblock
\APACrefbtitle {Real analysis} {Real analysis}\ (\BVOL~32).
\newblock
\APACaddressPublisher{}{Macmillan New York}.
\PrintBackRefs{\CurrentBib}

\bibitem [\protect \citeauthoryear {%
Shaked%
\ \BBA {} Shanthikumar%
}{%
Shaked%
\ \BBA {} Shanthikumar%
}{%
{\protect \APACyear {2007}}%
}]{%
shaked2007stochastic}
\APACinsertmetastar {%
shaked2007stochastic}%
\begin{APACrefauthors}%
Shaked, M.%
\BCBT {}\ \BBA {} Shanthikumar, J\BPBI G.%
\end{APACrefauthors}%
\unskip\
\newblock
\APACrefYear{2007}.
\newblock
\APACrefbtitle {Stochastic orders} {Stochastic orders}.
\newblock
\APACaddressPublisher{}{Springer}.
\PrintBackRefs{\CurrentBib}

\bibitem [\protect \citeauthoryear {%
Stein%
\ \BBA {} Shakarchi%
}{%
Stein%
\ \BBA {} Shakarchi%
}{%
{\protect \APACyear {2009}}%
}]{%
stein2009real}
\APACinsertmetastar {%
stein2009real}%
\begin{APACrefauthors}%
Stein, E\BPBI M.%
\BCBT {}\ \BBA {} Shakarchi, R.%
\end{APACrefauthors}%
\unskip\
\newblock
\APACrefYear{2009}.
\newblock
\APACrefbtitle {Real analysis: measure theory, integration, and Hilbert spaces} {Real analysis: measure theory, integration, and hilbert spaces}.
\newblock
\APACaddressPublisher{}{Princeton University Press}.
\PrintBackRefs{\CurrentBib}

\bibitem [\protect \citeauthoryear {%
Talagrand%
}{%
Talagrand%
}{%
{\protect \APACyear {1995}}%
}]{%
talagrand1995concentration}
\APACinsertmetastar {%
talagrand1995concentration}%
\begin{APACrefauthors}%
Talagrand, M.%
\end{APACrefauthors}%
\unskip\
\newblock
\APACrefYearMonthDay{1995}{}{}.
\newblock
{\BBOQ}\APACrefatitle {Concentration of measure and isoperimetric inequalities in product spaces} {Concentration of measure and isoperimetric inequalities in product spaces}.{\BBCQ}
\newblock
\APACjournalVolNumPages{Publications Math{\'e}matiques de l'Institut des Hautes Etudes Scientifiques}{81}{}{73--205}.
\PrintBackRefs{\CurrentBib}

\end{thebibliography}
\end{document}